\documentclass[a4paper,svgnames]{amsart}
\usepackage[hmarginratio={1:1},vmarginratio={1:1},lmargin=60.0pt,tmargin=60.0pt]{geometry}

\usepackage{amsmath,amsthm,amsfonts,amssymb,mathtools,float}
\usepackage{graphicx,scalerel} 
\usepackage{tikz-cd,tikz}
\usepackage{adjustbox}
\usepackage{enumitem}
\newenvironment{modularrep}{%
    \begin{tikzcd}[
            cells={nodes={draw=white,line width=4pt,
                          fill=black,circle,
                          minimum size=8pt,inner sep=0pt}},
            sep=4pt]
        }{
    \end{tikzcd}
}

\DeclarePairedDelimiter\floor{\lfloor}{\rfloor}
\DeclareMathOperator{\Hom}{Hom}
\newcommand{\C}{\mathbb{C}}
\newcommand{\R}{\mathbb{R}}

\newcommand{\N}{\mathbb{Z}_{> 0}}
\newcommand{\x}{\mathbf{x}}
\newcommand{\y}{\mathbf{y}}
\newcommand{\Q}{\mathbb{Q}}
\newcommand{\Z}{\mathbb{Z}}

\newcommand\scalemath[2]{\scalebox{#1}{\mbox{\ensuremath{\displaystyle #2}}}}

\numberwithin{equation}{subsection}
\renewcommand{\thesubsection}{\thesection\Alph{subsection}}

\newtheorem{Theorem}{Theorem}
\newtheorem{Proposition}[Theorem]{Proposition}
\newtheorem{Corollary}[Theorem]{Corollary}
\newtheorem{Lemma}[Theorem]{Lemma}

\theoremstyle{remark}
\newtheorem{Remark}{Remark}
\AtEndEnvironment{Remark}{\null\hfill$\Diamond$}%

\AtEndEnvironment{Question}{\null\hfill$\Diamond$}%

\theoremstyle{definition}

\AtEndEnvironment{Definition}{\null\hfill$\Diamond$}%
\newtheorem{Example}[Theorem]{Example}
\AtEndEnvironment{Example}{\null\hfill$\Diamond$}%

\usepackage[backref=pages]{hyperref}
\hypersetup{
  colorlinks = true,
  linkcolor =blue,
  anchorcolor = red,
  citecolor = blue,
  urlcolor = blue
}

\title{Growth Problems for Representations of Finite Groups}
\author{David He}
\address{D. He: School of Mathematics and Statistics, University of Sydney, Australia}
\email{dahe4726@uni.sydney.edu.au}

\begin{document}

\begin{abstract}
{We give a general asymptotic formula for the growth rate of the number of indecomposable summands in the tensor powers of representations of finite groups, over a field of arbitrary characteristic. In characteristic zero we obtain additional results, including an exact formula for the growth rate. We compute various examples and also provide code used to compute our formulas.}
\end{abstract}


\maketitle

\tableofcontents

\section{Introduction}
\subsection{Growth problems}
If $\mathcal{C}$ is an additive Krull--Schmidt monoidal category, following \cite{lacabanne2024asymptotics} we may define the \textit{growth problems} associated with its {additive} Grothendieck ring $K_0(\mathcal{C})$, which is a $\R_{\ge 0}$-algebra with a $\Z$-basis given by the isomorphism classes of indecomposable objects. The growth problems and related questions for various categories have been recently studied in the papers \cite{coulembier2023asymptotic,Coulembier_2023,lacabanne2023asymptotics, coulembier2024fractalbehaviortensorpowers, lachowska2024tensorpowersvectorrepresentation, lacabanne2024asymptotics}. Throughout this paper let $G$ denote a finite group, and $k$ an algebraically closed field. We continue the study of growth problems by specializing to the case where $\mathcal{C}=\text{Rep}_k(G)$ is the category of finite-dimensional $kG$-modules, and $K_0(\mathcal{C})$ is the representation ring. More precisely, for a $kG$-module $V$, we define the quantity $$b(n)=b^{G,V}(n)=\text{\# $G$-indecomposable summands in} \ V^{\otimes n } \ \text{(counted with multiplicity)}.$$
Then for us the growth problems associated with $V \in \text{Rep}_k(G)$ consist of the following questions:
\begin{enumerate}[label={(\arabic*)}]
    \item  \label{1} Can we find an explicit formula for $b(n)$?
    \item  \label{2} Can we find a nice \textit{asymptotic formula} $a(n)$ such that $b(n)\sim a(n)$? (Here $b(n)\sim a(n)$ if they are asymptotically equal: ${b(n)}/{a(n)}\xrightarrow{ n \to \infty } 1.$)
    \item \label{3}{Can we quantify the \textit{rate of convergence}, \textit{i.e.} how fast the quantity $\lvert b(n)/a(n)-1\rvert$ converges to 0?}
    \item \label{4} Can we bound the \textit{variance} $\lvert b(n)-a(n)\rvert$?
\end{enumerate}

We note that \ref{1} is much harder than \ref{2} and in general we cannot expect a closed formula for $b(n)$ to exist. However, in characteristic zero such a formula does exist and is given in Theorem \ref{exact formula}{; in this case we also obtain nice answers to questions \ref{3} and \ref{4}. Over a field of arbitrary characteristic $p\ge 0$, we will answer question \ref{2} in terms of the Brauer character table of $G$.}

\subsection{Main results}\label{char 0}

Suppose {$\text{char\ } k =p \ge 0$. Let $g_1,\dots, g_N$ be a complete set of representatives for the $p$-regular conjugacy classes of $G$ (all conjugacy classes, if $p=0$), and let $\chi_1,\dots,\chi_N$ be a complete set of irreducible (Brauer) characters of $G$. Denote by $Z_V(G)$ the subgroup of $G$ consisting of all $g\in G$ that act on $V$ by scaling, and denote by $\omega_V(g)$ the corresponding scalar. 
\begin{Theorem}\label{main asym}
Let $V$ be a faithful $kG$-module. 
\begin{enumerate}
\item the corresponding asymptotic growth rate is 
\begin{equation}\label{eqn:asym}
a(n)=\frac{1}{\lvert G\rvert}\sum_{\substack{1\le t\le N \\ g_t \in Z_V(G)}} S_{t}\big(\omega_V(g_t)\big)^n\cdot (\dim V)^n,   
\end{equation}
where $S_t$ is the sum over entries of the column corresponding to $g_t^{-1}$ in the (irreducible) Brauer character table.
\item We have $\lvert b(n)/a(n)-1\rvert \in \mathcal{O}(\lvert \lambda^{\mathrm{sec}}/\lambda \rvert^n + n^{-c})$, and $\lvert b(n)-a(n) \rvert \in \mathcal{O}(\lvert \lambda^{\mathrm{sec}}\rvert^n + n^{d})$ for some constants $c,d > 0.$ 
\end{enumerate}
\end{Theorem} 
\begin{Remark}
If $p\nmid \lvert G\rvert$, the Brauer character table coincides with the ordinary character table. The column sums of the ordinary character table are integers, unchanged by complex conjugation, and so $S_t$ is just the sum over entries of the column corresponding to $g_t$. It follows that when $p\nmid \lvert G\rvert$ our result is just \cite[(2A.1)]{lacabanne2023asymptotics} written in a different form. We therefore generalize \cite[(2A.1)]{lacabanne2023asymptotics}, which deals with the $p=0$ case and in turn generalizes \cite[Proposition 2.1]{coulembier2023asymptotic}.\end{Remark}
Recall that if $V$ is faithful, $Z_V(G)$ is a subgroup of the center $Z(G)$. We immediately obtain the following: \begin{Corollary}\label{p group} \
\begin{enumerate} 
    \item If $Z(G)$ is either trivial or a $p$-group (where $p> 0$ is the characteristic) and $V$ is faithful, then \begin{equation}\label{eqn: no period}
a(n)=\frac{1}{\lvert G \rvert}\sum_{t=1}^N \chi_t(1) \cdot (\dim V)^n.      
\end{equation}
\item In particular, if $G$ is a p-group and $V$ is faithful, then \begin{equation}\label{eqn: p group}
a(n)=\frac{1}{\lvert G\rvert} (\dim V)^n.    
\end{equation} 
\end{enumerate} 
\end{Corollary} }

{Let us assume now that $p\nmid \lvert G\rvert$ and $kG$ is semisimple.} The following results hold over any algebraically closed field and without loss of generality let $k = \C$. If $\chi_V$ is the character of $V$, let $\chi_{\text{sec}}$ denote any second largest character value of $\chi_V$ in terms of modulus (the largest being $\chi_V(1)=\dim V$). In the following result there is no requirement that $V$ be faithful: 
{\begin{Theorem}\label{exact formula}
 For a $\C G$-module $V$ with character $\chi$, the corresponding growth rate is $$b(n) = \frac{1}{\vert G\vert} \sum_{t=1}^N \vert C_t\vert S_t \big(\chi(g_t)\big)^n$$ where $S_t$ is the sum over entries in column of the character table corresponding to $g_t$. Moreover, with $a(n)$ as before, we have
 \begin{enumerate}
     \item {$\lvert b(n)/a(n)-1\rvert \in \mathcal{O}\big((\lvert {\chi_{\mathrm{sec}}}\rvert / {\dim V})^n\big)$. Here we use the usual capital O notation.} 
\item $\lvert b(n)-a(n)\rvert \in \mathcal{O}\big((\lvert \chi_{\mathrm{sec}}\rvert)^n\big)$. 
 \end{enumerate}
\end{Theorem}}

Now suppose $V$ is faithful, and denote the {periodic} expression preceding $(\dim V)^n$ in \eqref{eqn:asym} by $c_V(n)$, so that $a(n)=c_V(n)\cdot (\dim V)^n$. Clearly the dimension depends on $V$, but we will prove that $c_V(n)$ is independent of the particular $\C G$-module in the following sense:
\begin{Theorem}\label{independence}
If $V$ and $W$ are two faithful $\C G$-modules such that $Z_V(G)=Z_W(G)$, then $c_V(n)=c_W(n)$.   
\end{Theorem}

\begin{Remark}
The assumption that $Z_V(G)=Z_W(G)$ is necessary: {for example,} if $G=D_8$ is the dihedral group of order $8$ and $V= \C G$ is the regular representation, then $V$ is faithful, $Z_V(G)$ is trivial, and $a(n)=8^n$. However, for {any (two-dimensional)} faithful irreducible $kG$-module $W$ we have $a(n)=(3/4+(-1)^n/4) \cdot 2^n$ and so $c_V(n)\neq c_W(n)$.    
\end{Remark}

Recall that if $G$ has a faithful irreducible $\C G$-module $V$, then $Z(G)=Z_V(G)$ and $Z(G)$ is cyclic. Let $C_d$ denote the cyclic group of order $d$. Specializing Theorem \ref{independence} to faithful irreducible $\C G$-modules now gives the following:
\begin{Theorem} \label{irr faithful}
Let $V$ be a faithful irreducible $\C G$-module, and suppose $Z(G)\cong C_d, d\ge 1$. Then we have 
\begin{equation}\label{center roots}
a(n)=\Big(\frac{1}{\lvert G \rvert}\sum_{i=1}^d(\omega^i)^n S_{g^i}\Big) \cdot (\dim V)^n,    
\end{equation} where $\omega$ is any primitive $d$-th root of unity, $g$ is a generator of $Z(G)$, {and $S_{g^i}$ is the sum over the column corresponding to $g^i$}. Thus, if $V,W$ are two faithful irreducible $\C G$-modules, then $c_V(n)=c_W(n)$. Moreover, if $g^i$ and $g^j$ have the same order, then $S_{g^i}=S_{g^j}$.     
\end{Theorem}

\subsection{Examples}
To illustrate our results we discuss also various examples. In characteristic zero, we compute exact and asymptotic formulas for the dihedral groups, symmetric groups, $\text{SL}(2,q)$, and a family of semidirect products $C_{p^k}\rtimes C_{p^j}$, discussing also the {rate} of convergence and variance. In {modular} characteristic, we study {$\text{SL}(2,q)$, $\text{GL}(2,q)$, cyclic groups of prime order, and the Klein four group.} In the first two cases, we obtain a {formula} for $a(n)$. The last two cases are $p$-groups so the asymptotic formulas follow from {Corollary \ref{p group}}, and we obtain in addition {bounds on the rate} of convergence and variance for each indecomposable $kG$-module.  We have used Magma (\cite{bosma1997magma}) to obtain and verify our results. We explain how we have used Magma in the \hyperref[appendix]{Appendix}.\\

\noindent\textbf{Acknowledgments.}
The work here was undertaken by the author as an Honours research project at the University of Sydney under the supervision of Dani Tubbenhauer. The author thanks Dani for their generous support and guidance. {The author also thanks the referee for very helpful comments which have improved the quality of the paper, and for providing an insightful alternative proof of Theorem \ref{exact formula} (see Remark \ref{alt proof}).}  

\section{Proof of main results}\label{gen section}

\subsection{Exact and asymptotic formulas}
{We first recall the basic set-up for growth problems in $\text{Rep}_k(G)$}. For $V$ a $kG$-module, recall that the \textit{fusion graph} of $V$ is the (potentially countably infinite) directed and weighted graph whose vertices are the indecomposable $kG$-modules which appear in $V^{\otimes n}$ for some $n$, such that there is an edge with weight $m$ from the vertex ${V_i}$ to the vertex ${V_j}$ if ${V_j}$ occurs with multiplicity $m$ in $V\otimes {V_i}$. The corresponding adjacency matrix is called the \textit{action matrix}. If $V$ is faithful, the fusion graph of $V$ contains a unique strongly-connected component $\Gamma^p$ consisting of all projective indecomposable $kG$-modules, which we call the \textit{projective cell} ({following terminology in the proof of \cite[Proposition 4.22]{lacabanne2024asymptotics}).There is no path leaving $\Gamma^p$, since the tensor product of any $kG$-module with a projective $kG$-module is projective.} {We note that $\Gamma^p$ has only finitely many vertices.} Since $\Gamma^p$ is strongly-connected, its adjacency graph is a non-negative irreducible matrix with a well-defined period $h$, which we call the \textit{period} of $\Gamma^p$.  

\begin{Example}
If $k=\C$ and $\chi_1,\dots,\chi_N$ are all the irreducible characters of $G$, the action matrix ${M=M(V)}$ for a faithful $\C G$-module $V$ with character $\chi$ has entries $M_{ij}=\langle \chi  \chi_j,\chi_i\rangle$. In this case the projective cell is the whole fusion graph consisting of all irreducible $\C G$-modules. 
\end{Example}

As explained in the proof of \cite[Theorem 5.10]{lacabanne2024asymptotics}, if $h$ is the period of $\Gamma^p$, $M$ is the action matrix, and $\zeta=\exp(2\pi i h)$, then $\zeta^i(\dim V)$ is an eigenvalue of $M$ for $0\le i\le h-1$. These are all the eigenvalues of maximal modulus, {and they are all simple eigenvalues when $V$ is faithful. Note that clearly $h$ is bounded above by the number of simple $kG$-modules up to isomorphism.} 
 
{We are now ready to prove Theorem \ref{main asym}. As in the introduction let $g_1\dots, g_N$ be a complete set of representative for the $p$-regular conjugacy classes (all conjugacy classes if $k=0$), and let $\chi_1\dots \chi_N$ denote the irreducible (Brauer) characters. We fix an ordered basis $\begin{pmatrix}
V_1 & V_2 & \dots    
\end{pmatrix}$ for the action matrix $M(V)$ such that the first $N$ elements are the projective indecomposables $P_1=V_1,\dots, P_N=V_N$ where for $1\le t \le N$, $P_t$ is the projective cover of the module with character $\chi_t$. If $W$ is a $kG$-module, we denote its character by $\chi_W$.}
{\begin{Lemma}\label{correspondence lemma}
Let $V$ be a $kG$-module with (Brauer) character $\chi_V$, and let $M$ be the corresponding action matrix. Then for $1\le t \le N$, $\chi_V(g_t)$ is an eigenvalue of $M$, with corresponding right eigenvector $v_t:=\overline{\begin{pmatrix}
\chi_1(g_t) & \dots & \chi_N(g_t) & 0 & \dots    
\end{pmatrix}}$ and left eigenvector $w_t^T:=\begin{pmatrix}
\chi_{P_1}(g_t) & \dots & \chi_{P_N}(g_t) & \chi_{V_{N+1}}(g_t) & \dots     
\end{pmatrix}$   
\end{Lemma}
\begin{proof}
If we denote by $\tilde{M}$ the restriction of $M$ to the projective cell, then the entry $\tilde{M}_{ij}$ is the number of times $P_i$ occurs as a summand in $V\otimes P_j$, which is $\dim \Hom_{kG}(V\otimes P_j,S_i)=\langle \chi_V \chi_{P_j},\chi_{S_i}\rangle$. Using column orthogonality of Brauer characters, it is now easy to check that $v_t$ is a right eigenvector corresponding to $\chi_V(g_t)$. To see that $w_t^T$ is a left-eigenvector, notice that the $j$-th entry of $w_t^TM$ is $\sum_{i=1}^\infty \chi_i(g_t)M_{ij}=\chi_{V\otimes V_j}(g_t)=\chi_V(g_t)\chi_{j}(g_t)$. 
\end{proof}}

\begin{proof}[Proof of Theorem \ref{main asym}]
{By \cite[Theorem 5.10]{lacabanne2024asymptotics} and Lemma \ref{correspondence lemma}, $a(n)$ is sum over the column corresponding the trivial $kG$-module of the matrix  $$\frac{1}{\lvert G\rvert}\sum_{\substack{1\le t\le N \\ g_t \in Z_V(G)}} \big(\chi_V(g_t)\big)^n v_tw_t^T, $$
as the $\chi_V(g_t)$'s in the summation (which are equal to $\omega_V(g_t)\dim V$) are the eigenvalues of maximal modulus, and $w_t^Tv_t=\lvert G\rvert$ so we normalize by this factor. Part 1) of Theorem \ref{main asym} now follows by direct calculation, and part 2) follows from \cite[Theorem 5.10]{lacabanne2024asymptotics}.} 
\end{proof}

{\begin{proof}[Proof of Corollary \ref{p group}]
If $Z(G)$ is either trivial or a $p$-group, by Lemma \ref{correspondence lemma} there is only one eigenvalue of maximal modulus, so $a(n)$ is aperiodic. If $G$ is a $p$-group, the trivial module is the only irreducible $kG$-module, so the result follows.     
\end{proof}}

\begin{Remark}
If the action matrix is finite, then by \cite[Theorem 7]{lacabanne2023asymptotics} we have that $\lambda^{\mathrm{sec}}=0$ implies $b(n)=a(n)$. The presence of constant $d$ in the part 2) of Theorem \ref{main asym} suggests that this is false in general. For a counterexample in the infinite case, see \cite[Example 5.16]{lacabanne2024asymptotics}. See also Remark \ref{infinite sec}.
\end{Remark}

{\begin{Remark}
Lemma \ref{correspondence lemma} tells us that when $p\nmid \lvert G\rvert$ and $\Gamma^P=\Gamma$, the second largest eigenvalue coincides with the second largest character value, so in this case information about the rate of convergence and variance can also be obtained from the character table. However, in general this is not possible, see Example \ref{irr counterexample}.   
\end{Remark}}

\begin{Remark}
We do not lose any generality by requiring $V$ to be faithful, as we can always pass to a smaller quotient group which acts faithfully on $V$. In characteristic zero, the character table of $G$ `contains' the character table of its quotient groups: if $N$ is a normal subgroup of $G$, the irreducible characters of $G/N$ are in bijection with the irreducible characters $\chi$ of $G$ with $\chi(n)=\chi(1)$ for all $ n\in N$. As an example, the dihedral group of order 12,  $D_{12}\cong C_2\times S_3$, has a non-faithful two-dimensional irreducible character $\rho$ which becomes faithful once passed to the quotient group $S_3$, and the character table of $S_3$ may be obtained from that of $D_{12}$ by taking  only the rows $\chi$ with $\chi(g)=\chi(1)$ for all $g \in \ker \rho$, and collapsing the redundant columns.  
\end{Remark}

{If $k=\C$ (or more generally $p \nmid \lvert G \rvert$), the projective cell is the whole fusion graph, so Lemma \ref{correspondence lemma} in fact gives an orthogonal diagonalization of the action matrix $M$. Note that in this case $w_t^T=\overline{v_t}^T$.}  

{\begin{proof}[Proof of Theorem \ref{exact formula}]
The matrix $P_t:={\lvert C_t\rvert }/{\lvert G \rvert} \cdot \overline{v_t}v_t^T$ is the projection onto the span of $v_t$. From eigendecomposition, we have $$M^n=\sum_{t=1}^N \chi(g_t)^nP_t = \frac{1}{\lvert G \rvert}\sum_{t=1}^N \chi(g_t)^n \lvert C_t\rvert \overline{v_t}v_t^T.$$ 
Recalling (from e.g. \cite[Theorem 4]{lacabanne2023asymptotics}) that $b(n)$ is the sum over the column of $M^n$ corresponding to the trivial module, it follows that $$b(n)=\frac{1}{\vert G\vert} \sum_{t=1}^{{N}} \vert C_t\vert \overline{S_t} (\chi(g_t))^n,$$
where $\overline{S_t}$ is just the sum of the entries in the first column of $\overline{v_t}v_t^T$. Finally, $S_t=\overline{S_t}$ since the sum over a column of the character table is an integer. This last fact is well-known and also follows from the proof of Theorem \ref{independence} below, which shows that a column sum is fixed by the Galois group of a Galois extension over $\Q$ (a column sum is also an algebraic integer, so it is an integer). 
The statemens about rate of convergence and variance follow from Theorem \ref{main asym} since by Lemma \ref{correspondence lemma} in this case the eigenvalues of $M$ are the character values.   
\end{proof}
\begin{Remark} \label{alt proof}
We thank the referee for the insightful observation that Theorem \ref{exact formula} can also be proven using only character theory: we have $$b(n)=\sum_{t=1}^N \langle \chi_t, \chi^n\rangle = \frac{1}{\lvert G\rvert} \sum_{1\le t ,s \le N} \lvert C_s\rvert \overline{ \chi_t(g_s)} \big(\chi(g_s)\big)^n = \frac{1}{\lvert G\rvert} \sum_{t=1}^N \lvert C_t\rvert \overline{S_t}\big(\chi(g_s)\big)^n. $$ This is basically the same calculation.    
\end{Remark}
\begin{Remark} \label{remark caution}
Recall that a function $f: \Z_{\ge 0} \to \R$ converges geometrically to $a \in \R$ with ratio $\beta \in [0,1)$ if for all $\gamma \in (\beta, 1),\{(f(n)-a)/\gamma^n\}_{n \ge 0}$ is bounded, and we call the infimum of all ratios $\beta$ the \textbf{ratio of convergence}.
The statement that $\lvert b(n)/a(n)-1\rvert \in \mathcal{O}\big((\lvert {\chi_{\mathrm{sec}}}\rvert / {\dim V})^n\big)$ can be rephrased as saying that the ratio of convergence of $b(n)/a(n)\to 1$ is bounded above by $\lvert {\chi_{\mathrm{sec}}}\rvert / {\dim V}$. The two values are equal precisely when in the (simplified) expression of $b(n)$ given in Theorem \ref{exact formula}, the coefficient in front of at least one $(\chi_{\text{sec}})^n$ is nonzero. (Two things could go wrong: either all second largest character values correspond to columns whose entries sum to 0, or there could be cancellations.) For an example where the two values are distinct, take $V$ to be a faithful irreducible representation of $D_{16}$ (the dihedral group of order 16, \textit{c.f.} Theorem \ref{dihedral}) then we have $\chi_{\mathrm{sec}}=\sqrt{2}$, but $b(n)=a(n)$, \textit{i.e.} the ratio of convergence is 0.     
\end{Remark}}

\subsection{Independence of asymptotic formula}
We now prove Theorem \ref{independence}, which
is not immediately obvious if $Z_V(G)=Z_W(G)$ is nontrivial: since $V$ and $W$ are both faithful, it is clear that the sets of roots of unity $\{\omega_V(g): g \in Z_V({G})\}$ and $\{\omega_W(g):g \in Z_W({G})\}$ coincide, but \textit{a priori} the same root could be attached to different column sums in the two asymptotic formulas.

\begin{proof}[Proof of Theorem \ref{independence}]
To prepare for the proof, recall (\textit{e.g.} from \cite{MR2500560}) that there is an action of $\text{Gal}(\Q[\zeta_m]:\Q)$ on characters of $\C G$-modules, where $m$ is the exponent of $G$ and $\zeta_m$ is a primitive $m$-th root of unity. For $\sigma \in \text{Gal}(\Q[\zeta_m]:\Q)$, $\sigma$ acts on the character $\chi_V$ corresponding to $V$ by $\sigma \cdot \chi_V=\chi_{\sigma V}$, where $\sigma V$ is obtained by first realising the action of $V$ as matrices and then applying $\sigma$ entrywise. Call this the `row' action of ${\text{Gal}(\Q[\zeta_m]:\Q)}$, as it permutes the rows of the character table. There is also a `column' action of $\text{Gal}(\Q[\zeta_m]:\Q)$ on conjugacy classes of $G$: if $\sigma=\sigma_l$ corresponds to the automorphism taking $\zeta_m$ to $\zeta_m^l$, then $\sigma_l$ acts by $\sigma_l\cdot g = g^l$. Moreover, the two actions are compatible in the sense that 
\begin{equation} \label{galois commute}
(\sigma_l \cdot \chi_V) (g)=\chi_V(g^l)=\chi_V(\sigma_l \cdot g).  
\end{equation}
For a conjugacy class corresponding to $g$, the column sum $S_g$ is fixed by the row action of $\text{Gal}(\Q[\zeta_m]:\Q)$, and so by (\ref{galois commute}), we have $S_g=S_h$ if the conjugacy classes of $g$ and $h$ differ by a column action of $\text{Gal}(\Q[\zeta_m]:\Q)$. 

{Now we can write $Z_V(G)=Z_W(G)$ as $\{1, g, g^2,\dots, g^{d-1} \}$ where $d$ is the order of $g$. We can do this because $V$ being faithful means $Z_V(G)$ embeds via the map $g\mapsto \omega_V(g)$ into $\C$ as a multiplicative subgroup, and so is cyclic. Since $V, W$ are faithful, $\omega_V(g)$ and $\omega_W(g)$ are primitive $d$-th roots of unity. By discussions above, to show $c_V(n)=c_W(n)$ it suffices to show that if $\omega_V(g^x)=\omega_W(g^y)$ for some $1\le x,y < d $, then $g^x$ and $g^y$ differ by some column action. However, if $\omega_V(g^x)=\omega_W(g^y)$, then $\omega_V(g^x)$ and $\omega_V(g^y)$} must have the same order, so we can find some field automorphism $\sigma_l \in \text{Gal}(\Q[\zeta_m]:\Q)$ taking one root of unity to the other. Thus, $\sigma_l$ is the desired column action and we are done.  
\end{proof}

\begin{proof}[Proof of Theorem \ref{irr faithful}]
As discussed in the introduction, if $G$ has a faithful irreducible $\C G$-module $V$ then $Z(G)\cong C_d$ for some $d\ge 1$, and $Z_V(G)=Z(G)$. It follows that we can write $a(n)$ as in \eqref{center roots}, where $\omega$ is \textit{any} primitive $d$-th root of unity, since the (column) Galois action transitively permutes the elements in $Z(G)$ which have the same order (this is true because $Z(G)$ is cyclic). {Moreover, if $g^i$ and $g^j$ have the same order, then it follows from the proof of Theorem \ref{independence} above that $S_{g^i}=S_{g^j}$.}     
\end{proof}

\begin{Remark}
A group $G$ has a faithful irreducible $\C G$-module if and only if the direct product of all minimal normal subgroups of $G$ is generated by a single class of conjugates in $G$. This is due to \cite{gaschutz1954endliche}.    
\end{Remark}

\section{Examples in characteristic zero}

To illustrate our results we now apply them to study various examples.  

\subsection{Dihedral groups}
For $m\ge 3$ let $G=D_{2m}$ be the dihedral group of order $2m$. The asymptotic formulas for a faithful irreducible $\C D_{2m}$-module were obtained in \cite[Example 2]{lacabanne2023asymptotics} as
\begin{equation}\label{eqn:asym dihedral}
a(n)=\begin{cases}
 \big(\frac{m+1}{2m}\big) \cdot 2^n & m \text{ odd},\\
 \big(\frac{m+2}{2m}\big) \cdot 2^n & m \text{ even, $m/2$ odd},\\
 \big(\frac{m+2}{2m}+(-1)^n\frac{1}{m}\big) \cdot 2^n & m \text{ and $m/2$ both even}.
\end{cases}    
\end{equation}
By Theorem \ref{independence}, this formula remains true (after replacing $2^n$ with $(\dim V)^n$) for any faithful $\C G$-module with $Z_V(G)=Z(G)$, not necessarily irreducible. We now extend \cite[Example 2]{lacabanne2023asymptotics} by using Theorem \ref{exact formula} to compute also the exact formula, and characterize {$\chi_{\mathrm{sec}}$ and the ratio of convergence}.
\begin{Theorem}[Dihedral groups] \label{dihedral}
Let $V$ be a faithful irreducible  $\C D_{2m}$-module, then the exact growth rate is \begin{equation}\label{eqn: dihedral exact}
b(n) = \begin{cases}
\frac{2^n}{2m} \cdot \left(m+1 + 2\sum_{k=1}^{(m-1)/2} \cos^n(\frac{2\pi k}{m})\right) & \text{\ m odd},\\
\frac{2^n}{2m} \cdot \left( m+2 + 4\sum_{k=1}^{({m-2})/{4}} \cos^n(\frac{4\pi k}{m})\right) & \text{\ m even and \ } m/2 \text{\ odd}, \\
\frac{2^n}{2m} \cdot \left( m+2 + 2\cdot (-1)^n + 4\sum_{k=1}^{{\floor{({m-2})/{4}}}} \cos^n(\frac{4\pi k}{m})\right) & \text{\ m even and \ } m/2 \text{\ even},
\end{cases}    
\end{equation}
{where when $m=4$ the summation of cosines is taken to be 0.}
Moreover, a second largest character value is $\chi_{\mathrm{sec}} =2\cos\big({2m'\pi}/{m}\big)$ where $m'=(m+1)/2$ if $m$ is odd, and $m'=m/2+1$ if $m$ is even. {The ratio of convergence of $b(n)/a(n) \to 1$ is $\cos\big({(m+1)\pi}/{m}\big)$ if $m$ is odd, and $\cos\big((4\floor{(m-2)/4 }\pi/m\big)$ if $m>4$ is even. When $m=4$, the ratio of convergence is 0.}
\end{Theorem}
\begin{proof}
The proof is by straightforward application of Theorem \ref{exact formula}: let $\chi$ be the character of $V$, then the nonzero $\chi(g)$'s are all cosines, as given in \textit{e.g.} \cite[Section 5.3]{serre}. The statement about $\chi_{\mathrm{sec}}$ can likewise be obtained from the character table. {The ratio of convergence here corresponds to the ratio $\mu/\dim V$ where $\mu$ is any second largest character value of $V$ \textit{with nonzero column sum}, \textit{c.f.} Remark \ref{remark caution}. Note that $\mu$ sometimes coincide with $\chi_{\mathrm{\sec}}$ but not always. }   
\end{proof}

If we ignore the summations of cosines (which are column sums for columns outside the center) in \eqref{eqn: dihedral exact}, we recover the formulas for $a(n)$ given in \eqref{eqn:asym dihedral}. We note also that in this case $b(n)$ and $\chi_{\mathrm{sec}}$ are independent of the particular faithful irreducible $\C G$-module. In general, we can only expect $c_V(n)$ to be independent.

\begin{figure}[h]
    \centering
    \begin{minipage}{0.40\textwidth}
        \centering
\includegraphics[width=0.9\textwidth]{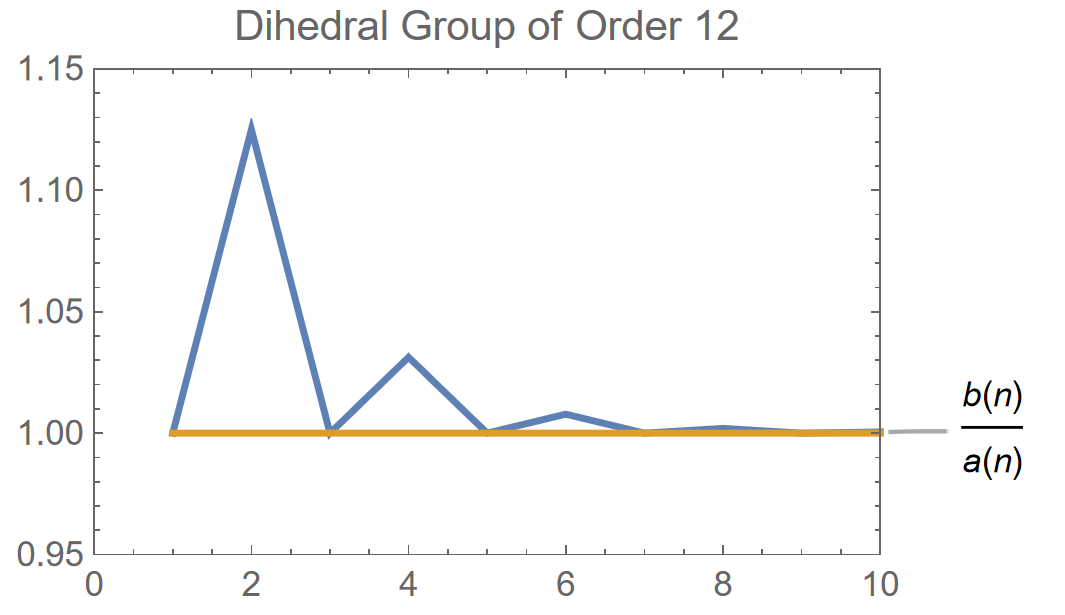}
    \end{minipage}
    \begin{minipage}{0.4\textwidth}
        \centering
    \includegraphics[ width=0.9\textwidth]{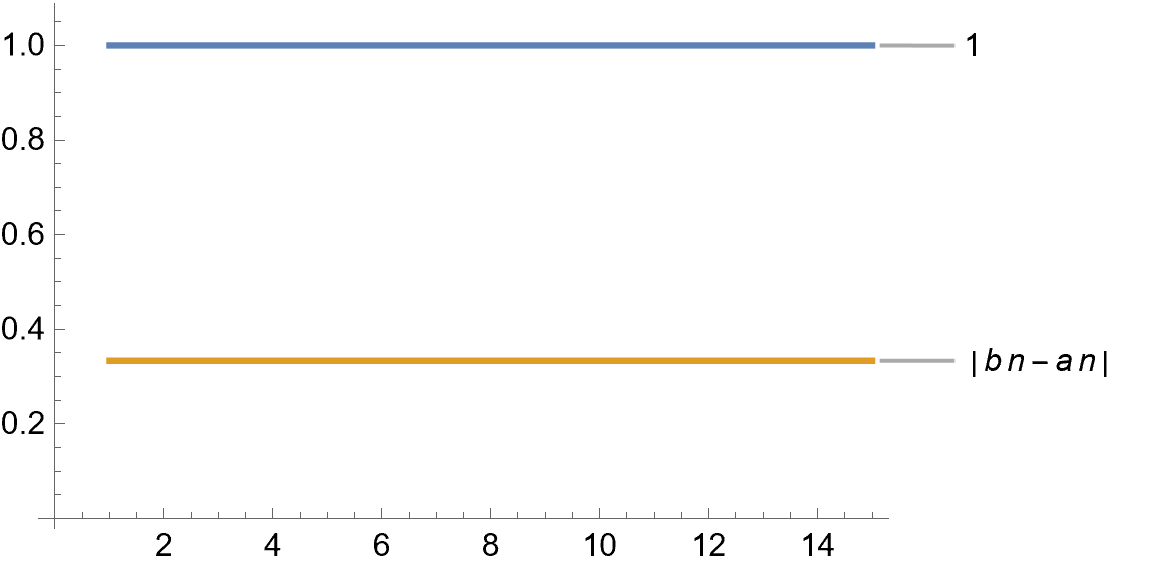}   
    \end{minipage}
     \caption{\textbf{Left}: the ratio $b(n)/a(n)$ for faithful irreducible $\C D_{12}$-modules. \textbf{Right}: the variance $\lvert b(n)-a(n)\rvert$.}
     \label{fig: dihedral}
\end{figure}

We illustrate the geometric convergence $b(n)/a(n) \to 1$ and variance for the $m=6$ case in Figure \ref{fig: dihedral}. In both graphs the $x$-axis represents $n$. In the graph on the left, the blue line represents the ratio $b(n)/a(n)$, which converges to the yellow line $y=1$ as we expect. In the graph on the right, the orange line represents the variance, while the blue line is $y=(\chi_{\mathrm{sec}})^n=1$. Since in this case $\chi_{\mathrm{sec}}=1$, as we expect from Theorem \ref{asym theorem} the variance is in $\mathcal{O}(1)$. In fact, $\lvert b(n)-a(n)\rvert = 1/3$.

\subsection{Symmetric groups}
The asymptotic formula for the symmetric groups $G= S_m$, $m\ge 3$ was calculated in \cite[Example 2.3]{coulembier2023asymptotic} to be \begin{equation}\label{eqn: sym asym}
a(n)= \sum_{z=0}^{\floor * {m/2}} \frac{1}{(m-2z)!z!2^z}\cdot (\dim V)^n.    
\end{equation}

Note that since $S_m$ has trivial center, the formula is true for any faithful $\C S_m$-module. We now extend the result to give an exact formula. By $a=(a_1,\dots,a_m)$, we mean the cycle type with $a_l$ different $l$-cycles, for $1\le l\le m$.  

\begin{Theorem}[Symmetric groups]
Let $V$ be a $\C S_m$-module with character $\chi$. The exact growth rate of $V$ is 
$$b(n) = \sum_{a} \frac{1}{\prod_{l={1}}^m l^{a_l} a_l!}r_2(a) \left(\chi(a)^n\right), \text{with\ } r_2(a)=\prod_{l\ge 1, a_l \neq 0} a_l! \varepsilon(l,a_l), $$ where the sum in $b(n)$ is over all cycle types $a=(a_1,\hdots,a_m),$ and 
$$
    \varepsilon(l,a_l) =    
    \begin{dcases*}
         \sum_{z=0}^{\floor*{a_l/{2}}}\frac{l^{z}}{(a_l-2z)! z!2^z} & $l$ odd, \vspace{0.3em}\\
    0 & $l$ even and $a_l$ odd, \\  
     \cfrac{l^{{a_l}/{2}}}{2^{{a_l}/{2}}\left({a_l}/{2}\right)!} & $l$ even and $a_l$ even.
    \end{dcases*}
$$    
\end{Theorem}

\begin{proof}
To apply Theorem \ref{exact formula} we need to sum over all the columns of the character table. Recall that all $\C S_m$-modules have Frobenius--Schur indicator equal to 1: by \textit{e.g.} \cite[Theorem 1]{Wigner}, in this case the column sum $S_t$ corresponds to the number of elements in $G$ squaring to $g_t$ (it is easy to see that this does not depend on the representative we choose). An explicit count for the number of $k$-th roots of a permutation is given in \cite[Theorem 1]{roots}. Specializing to the $k=2$ case, we obtain that the number of square roots of any $\alpha$ in the cycle class $a$, denoted $r_2(a)$, is given as in the statement of the Theorem. The exact formula now follows from Theorem \ref{exact formula}.     
\end{proof}

 We recover \eqref{eqn: sym asym} if in the summation for $b(n)$ we restrict to the term corresponding to the identity, with $a=(m,0
,\dots,0).$

The ratios ${b(n)}/{a(n)}$ for all faithful irreducible $\C S_5$-modules are plotted in Figure \ref{fig: sym figure} beside the character table of $S_5$. As we expect from Theorem \ref{asym theorem}, we see that the $\C S_5$-modules of dimension 6 and 4 converge slower ({the ratios of convergence are 1/3 and 1/4 respectively)} than that of dimension 5 ({the ratio of convergence is 1/5}). By Theorem \ref{asym theorem} we also have $\lvert b(n)-a(n)\rvert \in \mathcal{O}\big((\chi_{\mathrm{sec}})^n\big)$, so for the 5-dimensional irreducible $\C S_5$-module we have  $\lvert b(n)-a(n)\rvert \in \mathcal{O}(1)$ and is a constant. For the six-dimensional irreducible $\C S_5$-module we have $\lvert b(n)-a(n)\rvert \in \mathcal{O}(2^n)$, which is illustrated in Figure \ref{fig:s5 var}.  

\begin{figure}[H]
    \centering
    \begin{minipage}{0.40\textwidth}
        \centering
\includegraphics[width=0.9\textwidth]{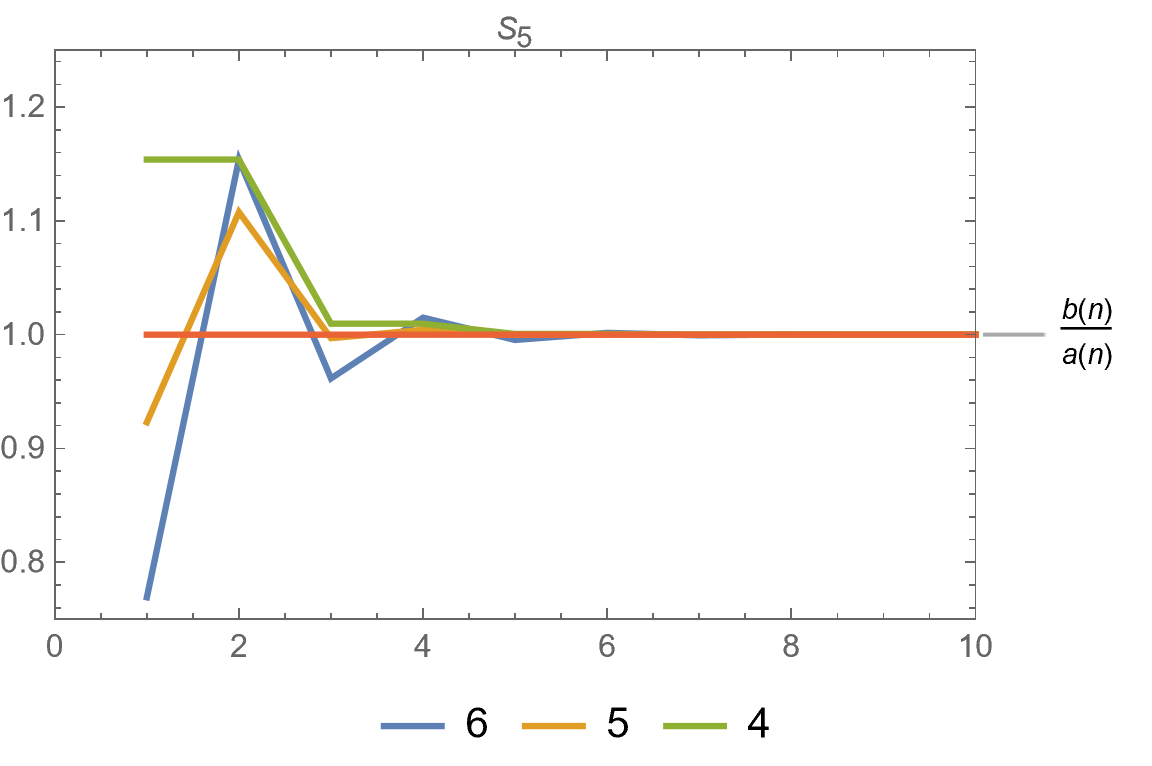}
    \end{minipage}
    \begin{minipage}{0.4\textwidth}
        \centering
    \includegraphics[width=0.6\textwidth]{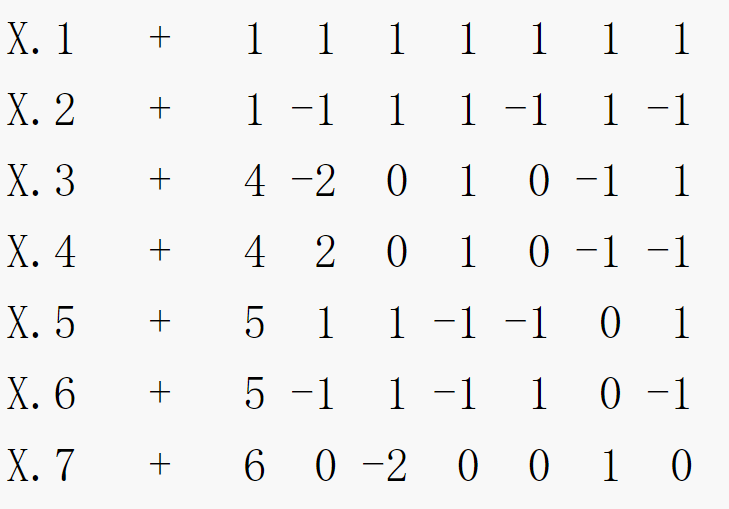}   
    \end{minipage}
     \caption{\textbf{Left}: the ratio $b(n)/a(n)$ for faithful irreducible $\C S_5$-modules, labelled with dimensions. \textbf{Right}: the character table for $S_5$ obtained from Magma.}
     \label{fig: sym figure}
\end{figure}

\begin{figure}[H]
    \centering
\includegraphics[width=0.5\linewidth]{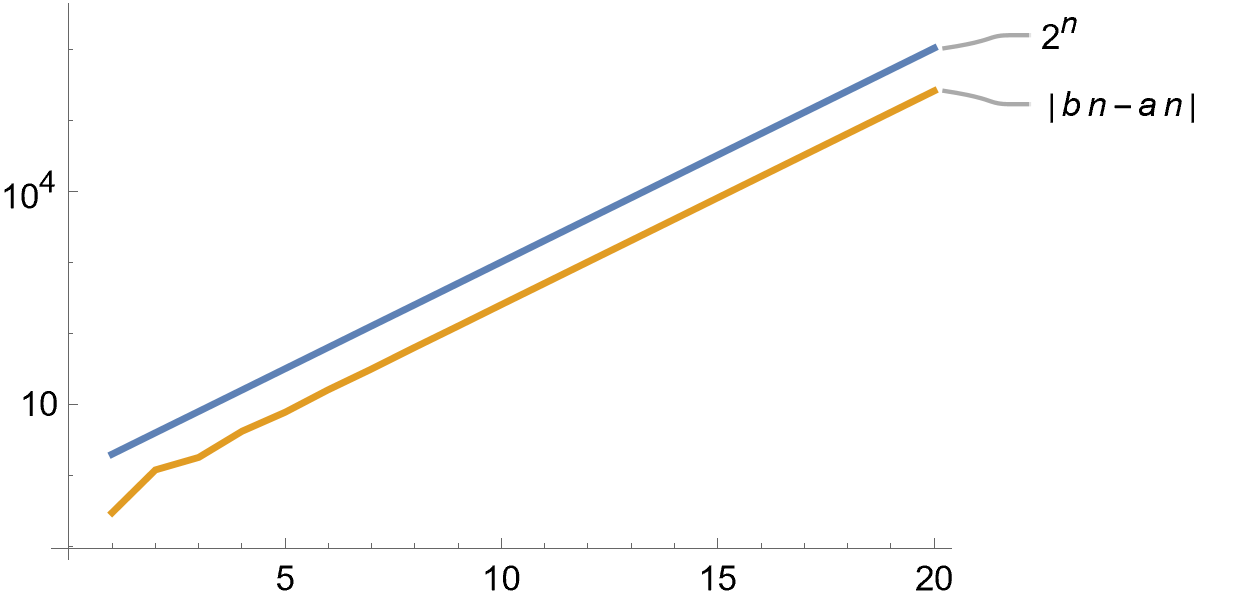}
    \caption{Variance for the six-dimensional irreducible $\C S_5$-module. The $y$-axis uses a logarithmic scale.}
    \label{fig:s5 var}
\end{figure}

\subsection{Special linear groups $\textnormal{SL}(2,q)$}
We now compute the exact and asymptotic formulas for the special linear groups $G=\textnormal{SL}(2,q)$, where $q=p^m$ is some prime power. The second {largest} character values can also be obtained explicitly but varies depending on the particular family of irreducible $\C G$-module, so we do not write them down here.
Following the notation of \cite[Section 38]{dornhoff1971group}, let $a=\begin{psmallmatrix}
 \nu & 0 \\
 0 & \nu^{-1} 
\end{psmallmatrix}$ where $\langle \nu \rangle = \mathbb{F}_q^\times$, {let $z=\begin{psmallmatrix}
 -1 & 0\\
 0 & -1
\end{psmallmatrix}$, $c=\begin{psmallmatrix}
 1 & 0\\
 1& 1
\end{psmallmatrix}$, $d=\begin{psmallmatrix}
  1 & 0\\
  \nu & 1
\end{psmallmatrix}$,} and let $b$ be an element of order $q+1$. We begin with the odd case (\textit{i.e.} $p \neq 2$):

\begin{Theorem}[$\textnormal{SL}(2,q)$, odd case]
Let $V$ be a $\C G$-module, where $G=\textup{SL}(2,q)$ for some odd prime power $q$, then the exact growth rate is
{\begin{align*}
    b(n) =\frac{1}{q^3-q}\left((q^2+q) (\dim V)^n + \alpha  (\chi(z))^n\right) 
    + \frac{\alpha}{2q}((\chi(zc))^n+(\chi(zd))^n) + \\
    2\Big(\frac{1}{q-1}\sum_{l=1}^{\floor{(q-3)/4}}\big(\chi(a^{2l})^n\big)+\frac{1}{q+1}\sum_{m=1}^{\floor{(q-1)/4}}\big(\chi(b^{2m})\big)^n\Big),
\end{align*}
where $\alpha$ is 2 if $(q-1)/{2}$ is even, and 0 otherwise, and the summations are taken to be 0 if the upper index is smaller than 1. }

If $V$ is faithful and irreducible, then the asymptotic growth rate is
$$a(n)=\begin{cases}
\frac{1}{q^3-q}(q^2+q+2(-1)^n)\cdot (\dim V)^n & (q-1)/2 \ \text{even},\\
\frac{1}{q^3-q}(q^2+q)\cdot (\dim V)^n & (q-1)/2 \ \text{odd}.
\end{cases}$$    
\end{Theorem}
\begin{proof}
The formulas follow from Theorems \ref{exact formula} and \ref{asym theorem} via direct calculations using the character table of $\text{SL}(2,q)$, which can be found in \textit{e.g.} \cite[Section 38]{dornhoff1971group}.    
\end{proof}

The even case (\textit{i.e.} $p=2$) is analogous. In this case the center is trivial so we have no alternating behaviour in $a(n)$, and the asymptotic formula holds for all faithful $\C G$-modules.

\begin{Theorem}[$\textnormal{SL}(2,q)$, even case] Let $G=\textup{SL}(2,q)$ for $q=2^n$. For a $\C G$-module $V$ with character $\chi$, the exact growth rate is 
$$b(n)=\frac{1}{q^3-q}\Big(q^2 (\dim V)^n + q(q+1)\sum_{l=1}^{(q-2)/{2}}(\chi({a^{l}}))^n+q(q-1)\sum_{m=1}^{q/2}(\chi({b^{m}}))^n \Big),$$ and the asymptotic growth rate is 
$$a(n)=\frac{q}{q^2-1}(\dim V)^n.$$
    
\end{Theorem}

We plot the ratio ${b(n)}/{a(n)}$ for the faithful irreducible $\C G$-modules for the case $q=5$ in Figure \ref{fig: Sl odd}. The convergence is rather fast for the  $\C G$-modules of dimension 4 and 6: in this case {the ratios of convergence are} $1/4 = 0.25$ and $1/6 \approx 0.167$ respectively. The convergence for the two-dimensional $\C G$-modules ($\eta_1$ and $\eta_2$ {in the notation of \cite[Section 38]{dornhoff1971group}}) is much slower, as in this case {the ratio of convergence coincides with} ${\lvert \chi_{\mathrm{sec}}\rvert}/{\dim V} =\lvert \exp(4\pi i/5)+\exp(6\pi i/5)\rvert/{2} \approx 0.809$.

\begin{figure}[H]
    \centering
    \begin{minipage}{0.40\textwidth}
        \centering
\includegraphics[width=0.9\textwidth]{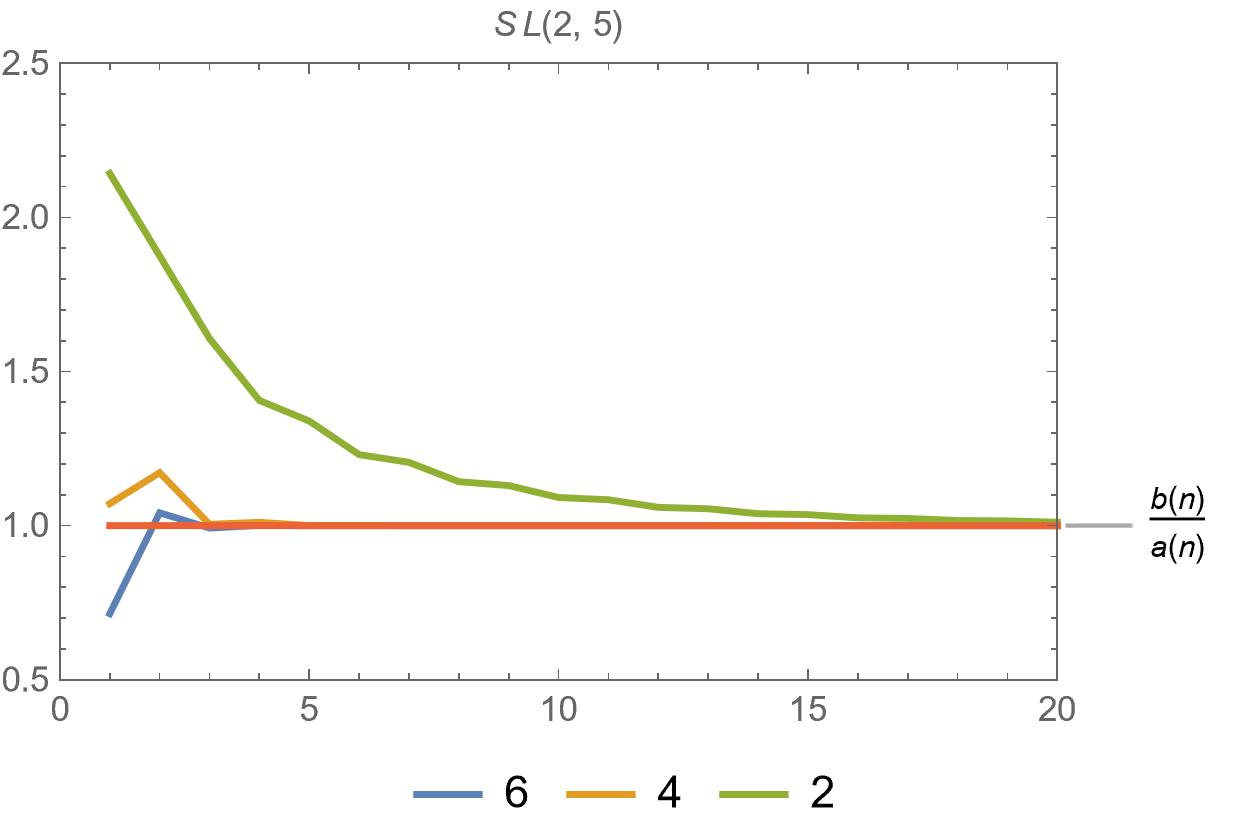}
    \end{minipage}
    \begin{minipage}{0.4\textwidth}
        \centering
    \includegraphics[width=0.9\textwidth]{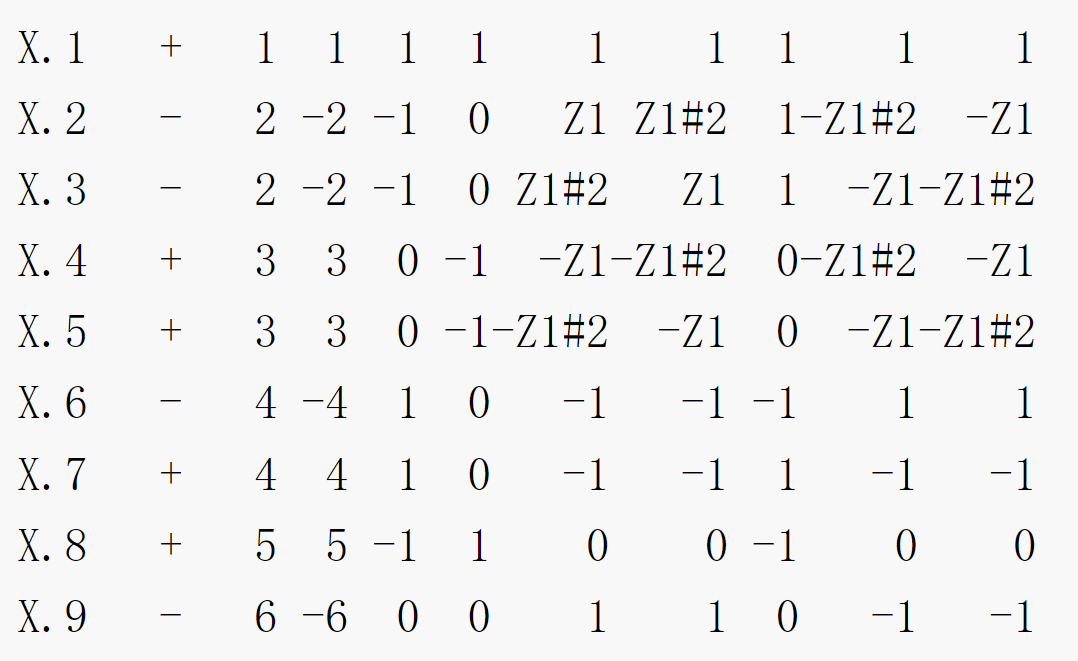}   
    \end{minipage}
     \caption{\textbf{Left}: The ratio ${b(n)}/{a(n)}$ for all faithful irreducible $\C \text{SL}(2,5)$-modules, labelled with their dimensions. \textbf{Right}: the character table for $\text{SL}(2,5)$ from Magma, where if we let $\omega = \exp({2\pi i / 5})$ the symbols $Z1,Z1\#2$ correspond to $\omega+\omega^4$ and $\omega^2+\omega^3$ respectively.}
     \label{fig: Sl odd}
\end{figure}

\subsection{A family of semidirect products}
In the examples discussed so far, the asymptotic formula $a(n)$ has period at most 2. Theorem \ref{irr faithful} suggests, however, that when the (cyclic) center of a group is large, we can expect increasingly complicated behaviour due to many roots of unity appearing in the asymptotic formula. We now demonstrate this by studying the family of groups $G=G(p,k,j)$ of order $p^{k+j}$, with presentation $\langle a,b \mid a^{p^k} = b^{p^j} = 1, bab^{-1} = a^{p^{k-j}+1} \rangle$, where $k,j \in \N, k-j\ge j$, and $p$ is a prime. We will denote such a group by $C_{p^k}\rtimes C_{p^j}$ though it may not be the only semidirect product.
\begin{Theorem} \label{semi}
Let $G=C_{p^k}\rtimes C_{p^j}$ and let $V$ be a faithful irreducible $\C G$-module. Let $[\omega_i^n]$ denote $\sum_{\omega} \omega^n$ where the sum is over all primitive $p^i$-th roots of unity. Then the growth rate for $V$ is 
 $$b(n)=a(n)= \frac{(p^j)^n}{p^{k+j}}\left((j+1)p^k-jp^{k-1}+(p^k-p^{k-1})\sum_{m=0}^{j-1} (m+1)[\omega_{j-m}^n]\right).$$    
\end{Theorem}

{We note that the formula appears to not depend on $\dim V$ because all faithful irreducible $\C$G-modules have dimension $p^j$.}
To illustrate:
\begin{enumerate}
    \item For $G=C_{32} \rtimes C_{4}$, $V$ a  faithful 4-dimensional  irreducible $\C G$-module, Theorem \ref{semi} gives $$b(n)=a(n)= \frac{4^{n}}{128}\Big(64+32(-1)^n+16(i^n+(-i)^n)\Big).$$ 
    \item For $G=C_{81} \rtimes C_{9}$, $V$ a  faithful 9-dimensional  irreducible $\C G$-module, Theorem \ref{semi} gives $$b(n)=a(n)=\frac{9^n}{729}\Big(189+108(\omega^{3n}+\omega^{6n})+54(\omega^n+\omega^{2n}+\omega^{4n}+\omega^{5n}+\omega^{7n}+\omega^{8n})\Big),$$
    where $\omega$ is a primitive $9$-th root of unity.
\end{enumerate}

The two cases are illustrated in Figures \ref{Fig:c32-4} and \ref{Fig: c81-9}, where we plot the value of $a(n)$ divided by the leading growth rate (\text{i.e.} by $4^n\cdot 64/128$ and $9^n\cdot 189/729$ respectively). The subsequential limits of $c_V(n)$ depending on the value of $n \bmod {p^j}$ is reflected in the periodic alternations in the plots. In particular, Theorem \ref{semi} shows that the potential number of subsequential limits of $c_V(n)$ is unbounded.

\begin{figure}[H]
    \centering
    \begin{minipage}{0.45\textwidth}
        \centering
        \includegraphics[width=0.9\textwidth]{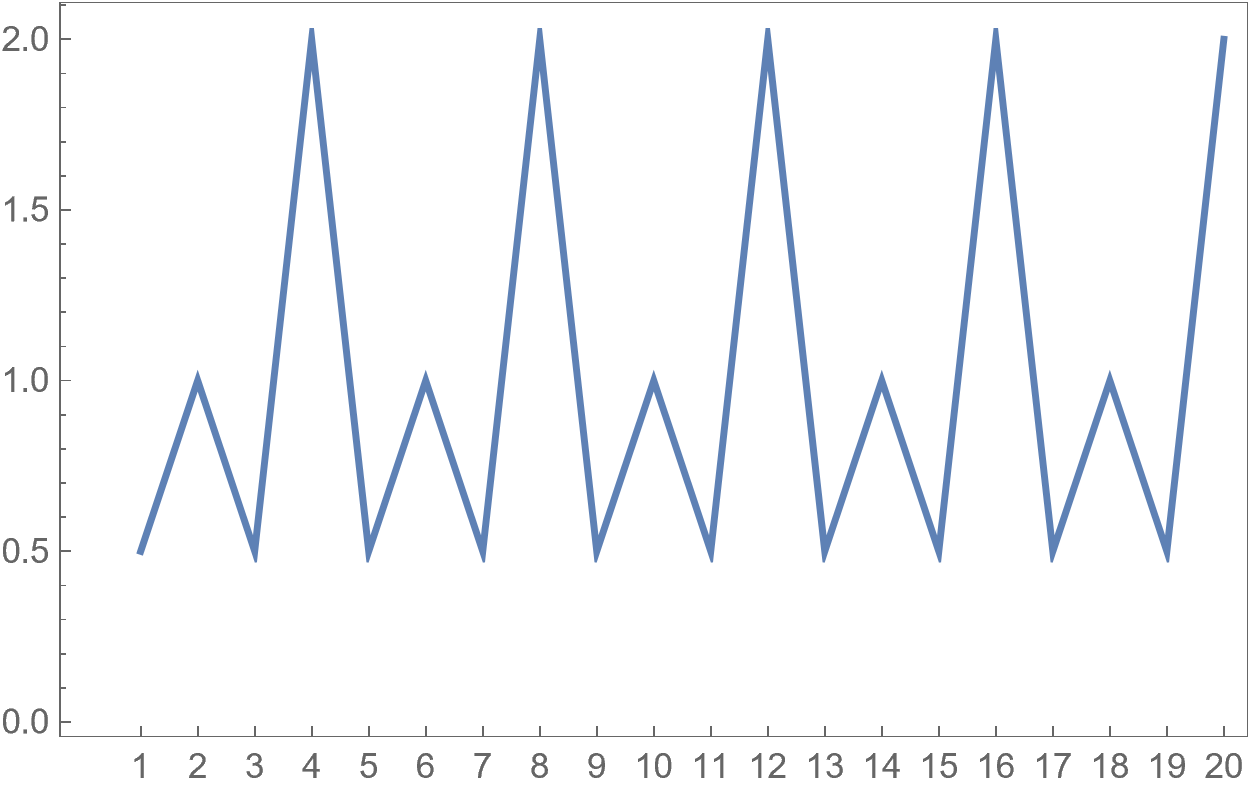}\caption{$C_{32}\rtimes C_4$}
        \label{Fig:c32-4}
    \end{minipage}\hfill
    \begin{minipage}{0.45\textwidth}
        \centering
    \includegraphics[width=0.9\textwidth]{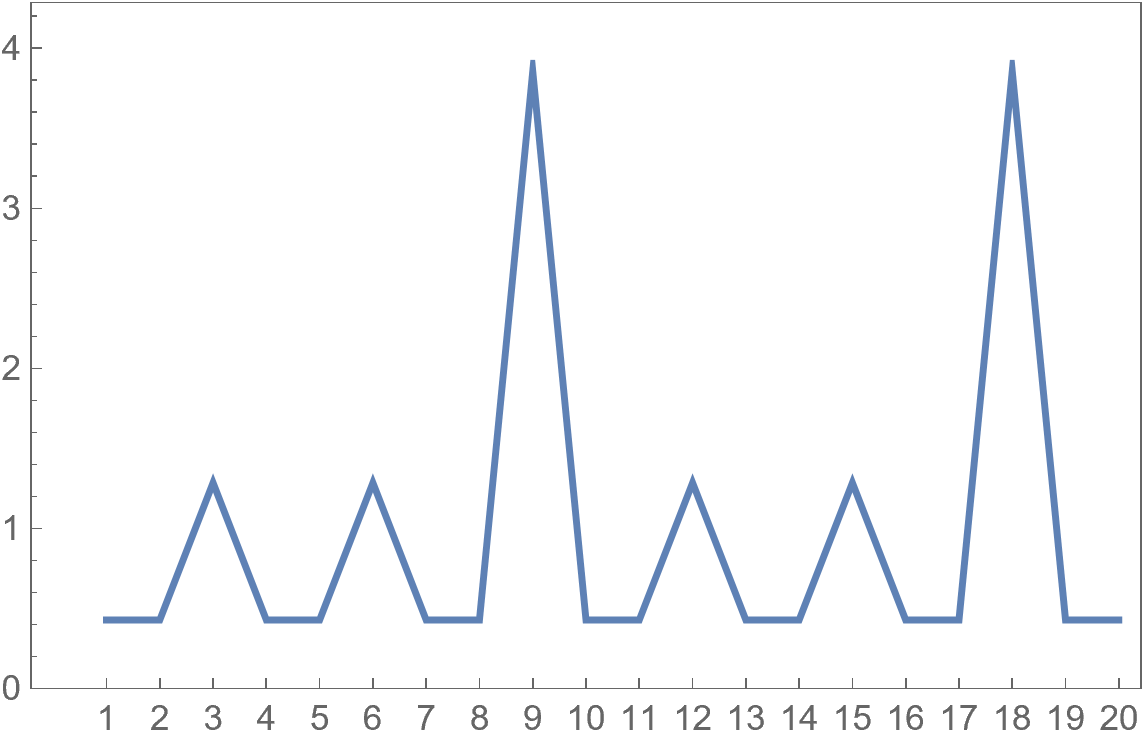}
        \caption{$C_{81}\rtimes C_9$}
     \label{Fig: c81-9}
    \end{minipage}
   
\end{figure}

\begin{proof}[Proof of Theorem \ref{semi}]
Let $G=C_{p^k}\rtimes C_{p^j}$. It can be verified that $Z(G)=\langle a^{p^j}\rangle \cong C_{p^{k-j}}$ and the irreducible $\C G$-modules are: 
\begin{enumerate}
    \item $p^k$ one-dimensional $\C G$-modules, defined by $\rho_{il}:a \mapsto \zeta^i$, $b\mapsto \alpha^l, 1\le i \le p^{k-j},1\le l\le p^j$, where $\zeta$ is a $p^{k-j}$-th root of unity and $\alpha$ is a $p^j$-th root of unity.  
    \item For $1 \le m \le j$, $p^{k-m-1}(p-1)$ distinct $\C G$-modules of dimension $p^m$, of the form $$\rho^m_{\eta,\delta}: a\mapsto  {\begin{pmatrix}
     \eta & 0 & \cdots & \cdots & 0\\
     0 & \eta^{p^{k-j}+1} & \cdots & \cdots & 0\\
     0 & 0 & \eta^{(p^{k-j}+1)^2} & \cdots & 0\\
     \vdots &\cdots & \cdots &\ddots & \vdots\\
     0 & \cdots & \cdots & \cdots& \eta^{(p^{k-j}+1)^{p^m -1}}
    \end{pmatrix}}, b\mapsto \begin{pmatrix}
       0 & \delta & 0 & \cdots & 0\\
       0 & 0 & 1 & \cdots & 0 \\
       \vdots & \vdots & \vdots & \ddots &  \vdots \\
       0 & \cdots & \cdots & \cdots & 1\\
       1 & 0 & \cdots & \cdots  & 0
    \end{pmatrix},$$
    where $\delta$ is a (not necessarily primitive) $p^{j-m}$-th root of unity and $\eta$ is a $p^{k-j+m}$-th primitive root of unity. {Note that different $\rho_{\eta,\delta}^m$'s may define isomorphic representations.}
\end{enumerate} 
{To see that all the $\rho_{\eta,\delta}^m$'s are irreducible, we note that the assumption $k-j\ge j$ guarantees that the diagonal entries of $\rho_{\eta,\delta}^m(a)$ are all distinct $p^{k-j+m}$-th roots of unity: we have \begin{align*}
(p^{k-j}+1)^l&=1+lp^{k-j}+\binom{l}{2}p^{2(k-j)}+\dots + p^{l(k-j)}\\
&\equiv 1+lp^{k-j} \pmod {p^{k-j+m}}
\end{align*} 
since $2(k-j)\ge k \ge {k-j+m}$, and so $(p^{k-j}+1)^l$ has multiplicative order $p^m$ modulo $p^{k-j+m}$. It follows that the trace of $\rho_{\eta,\delta}^m(a^x)$ vanishes when $p^m \nmid x$, as after factoring out $\eta$ the diagonal entries of $\rho_{\eta,\delta}^m(a^x)$ consist of distinct $p^m$-th roots of unity which sum to 0. Quick character calculations now confirm that $\rho_{\eta,\delta}^m$ is indeed irreducible. (If $k-j <j$, the $\rho_{\eta,\delta}^m$'s are not all guaranteed to be irreducible. For example, $C_{8}\rtimes C_4$ has no irreducible 4-dimensional representation.)\\
From characters we also see that $\rho_{\eta,\delta}^m \cong \rho_{\eta^{'}, \delta^{'}}^m$ if and only if $\eta'=\eta^{(p^{k-j}+1)^l}$ for some $0\le l \le p^m -1$  and $\delta=\delta'$. We deduce that there are $\phi(p^{k-j+m})/p^m \cdot p^{j-m} = p^{k-m-1}(p-1)$ distinct representations for each $m$. (Here $\phi$ is the Euler totient function.)}
The squares of the dimensions of the $\C G$-modules listed above sum to the group order, so we have found all the irreducibles. 

Character computation gives that the only faithful irreducible $\C G$-modules are those of dimension $p^j$. The characters for these modules vanish outside the center, so $b(n)=a(n)$. To calculate the column sums, we observe that the sum over all {distinct} $p^m$-dimensional characters evaluated at $a^x \in Z(G)$ is $p^{j-m}\sum_{\eta} \chi^m_\eta(a^x)=p^{j-m}C_{p^{k-j+m}}(x)$, where $\chi^m_{\eta}$ is the character of any $\rho^m_{\eta,\delta}$, and $C_{p^{k-j+m}}(x)$ is the Ramanujan sum over the $x$-th power of all primitive $p^{k-j+m}$-th roots of unity. Since central elements of the same order have the same corresponding column sum by Theorem \ref{irr faithful}, and the order of $a^x$ depends only on the number of times $p$ divides $x$, for {$p^z \le x < p^{z+1}$ we can treat $a^{x}$ as if it is $a^{p^{z}}$} when calculating the column sum. In fact, the column sum for $a^{p^{j+z}}$ vanish unless $z\ge k-2j$, so it suffices to compute $S_{a^y}$ for $y=k-j+w, 0\le w \le j$. Applying the formula for Ramanujan sum (see \textit{e.g} \cite[\S 3]{Hoelder1936}) gives $$S_m =\begin{cases}
0 & m > w+1,\\
-p^{k-1} & m = w+1,\\
(p-1)p^{k-1} & m < w+1,
\end{cases}$$
where, for fixed $y=k-j+w$, $S_m, 1\le m \le j$ denotes the sum over the $p^m$-dimensional characters evaluated at $a^y$. It is moreover easy to see that $S_0=p^k$, so $$S_{a^{y}}=p^k+\sum_{m=1}^{w} (p-1)p^{k-1}-p^{k-1}= (w+1)(p^k-p^{k-1}).$$
The statement of the theorem now follows from Theorem \ref{irr faithful} after observing that the character of {each} element $a^{lp^{j}}{\in Z(G)}$ where $p^{k-j+w} \le lp^j < p^{k-j+w+1}$, is $p^j$ multiplied by a distinct primitive $p^{j-w}$-th root of unity, and all primitive $p^{j-w}$-th roots of unity are obtained this way. {The assumption $k-j\ge j$ ensures that $Z(G)$ is big enough and we obtain all $p^j$ roots of unity in the formula.}
\end{proof}

\section{Nonsemisimple examples}
\subsection{Special and general linear groups $\textnormal{SL}(2,q)$ and $\textnormal{GL}(2,q)$}

{\begin{Theorem}
Let $G=\textup{SL}(2,q)$ or $\textup{GL}(2,q)$, where $q=p^r$ is a prime power, and let $\textup{char}\ k=p$. If $V$ is a faithful irreducible $kG$-module, then \begin{equation}\label{eqn: conjecture}
a(n)=\frac{(p+1)^r}{2^r(q+1)(q-1)}\Big(1+ \frac{1}{q}(-1)^n\Big)\cdot (\dim V)^n.    
\end{equation}
If $q$ is even, then there is no period, with {the formula as above but with the $(-1)^n$ term removed}.     
\end{Theorem}
\begin{proof}
We can applying Theorem \ref{main asym} to the modular character table of $\textup{SL}(2,q)$, which is worked out in \cite[\S 2]{Brauer}. The case for $\textup{GL}(2,q)$ is similar (in this case there are $(q-1)$ times more characters of each dimension coming from tensoring with some power of the one-dimensional determinant module, see \cite[\S 30]{Sri}, but the order of the group also increases by a factor of $(q-1)$ so $a(n)$ remains the same).    
\end{proof}
By \cite[Theorem B]{craven2013tensor}, the action matrix for an irreducible $kG$-module is actually finite. We compute some specific examples.} 

{\begin{Example} \label{irr counterexample}
 Let $G=\textup{SL}(2,7)$ and $\textup{char} \ k=7$.
 For the 4-dimensional irreducible $kG$-module we have 
 $$a(n)=\Big(\frac{1}{12}+\frac{1}{84}(-1)^n\Big) \cdot 4^n$$ and the eigenvalues are $\{\pm4,\pm(1+\omega+\omega^6),\pm(\omega^2+\omega^5+1),\pm(\omega^3+\omega^4+1),\pm1,0\}$, where $0$ has multiplicity 3, and $\omega=\exp(2\pi i/7)$. (See Section \ref{action magma} in the appendix for the code). Unlike in the characteristic zero case, these eigenvalues can no longer all be found in the (ordinary or modular) character table. In particular, the second largest eigenvalues $\pm (1+\omega+\omega^6)$ are not in the $\Z$-span of values in the character table of $G$, which is $\{a+b(\omega^2+\omega^4+\omega^7)+c(\zeta-\zeta^3) \mid a,b,c \in \Z\}$, where $\zeta = \exp(2\pi i/8)$.
  Here {the ratio of convergence coincides with} $\lvert \lambda^{\mathrm{sec}}\rvert/\dim V \approx 0.56.$ The fusion graph is plotted in Figure \ref{fig:sl(2,7)}, and the ratio $b(n)/a(n)$ is plotted in Figure \ref{fig:sl2,7 4d ratio}.
 \end{Example}
 \begin{figure}[H]
    \centering
    \begin{minipage}{0.40\textwidth}
    \centering
    \includegraphics[width=\textwidth]{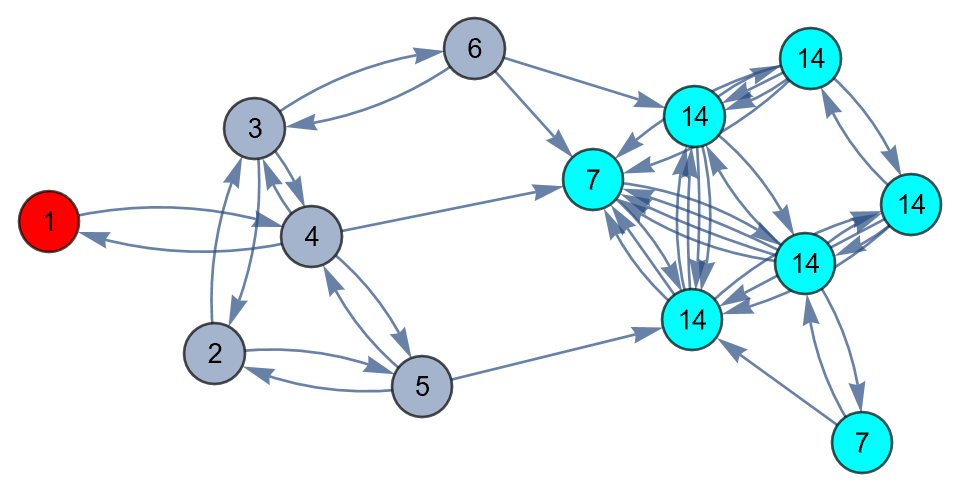} 
    \end{minipage}
    \begin{minipage}{0.4\textwidth}
     \centering
     $$\scalemath{0.6}{\begin{pmatrix}
0 & 1 & 0 & 0 & 0 & 0 & 0 & 0 & 0 & 0 & 0 & 0 & 0 \\
1 & 0 & 1 & 1 & 0 & 0 & 0 & 0 & 0 & 0 & 0 & 0 & 0 \\
0 & 1 & 0 & 0 & 1 & 1 & 0 & 0 & 0 & 0 & 0 & 0 & 0 \\
0 & 1 & 0 & 0 & 0 & 1 & 0 & 0 & 0 & 0 & 0 & 0 & 0 \\
0 & 0 & 1 & 0 & 0 & 0 & 0 & 0 & 0 & 0 & 0 & 0 & 0 \\
0 & 0 & 1 & 1 & 0 & 0 & 0 & 0 & 0 & 0 & 0 & 0 & 0 \\
0 & 1 & 0 & 0 & 1 & 0 & 0 & 2 & 3 & 0 & 0 & 0 & 2 \\
0 & 0 & 0 & 1 & 0 & 0 & 1 & 0 & 0 & 2 & 2 & 1 & 0 \\
0 & 0 & 0 & 0 & 0 & 0 & 1 & 0 & 0 & 1 & 1 & 1 & 0 \\
0 & 0 & 0 & 0 & 1 & 0 & 0 & 2 & 1 & 0 & 0 & 0 & 2 \\
0 & 0 & 0 & 0 & 0 & 0 & 0 & 1 & 1 & 0 & 0 & 0 & 1 \\
0 & 0 & 0 & 0 & 0 & 0 & 0 & 0 & 1 & 0 & 0 & 0 & 0 \\
0 & 0 & 0 & 0 & 0 & 0 & 0 & 0 & 0 & 1 & 1 & 0 & 0 \\
\end{pmatrix}}
$$
    \end{minipage}
    \caption{The fusion graph and action matrix for the four-dimensional irreducible $k\text{SL}(2,7)$-module in characteristic 7. The trivial and projective modules are colored red and cyan respectively.} 
    \label{fig:sl(2,7)}
\end{figure}
\begin{figure}[H]
    \centering    \includegraphics[width=0.4\textwidth]{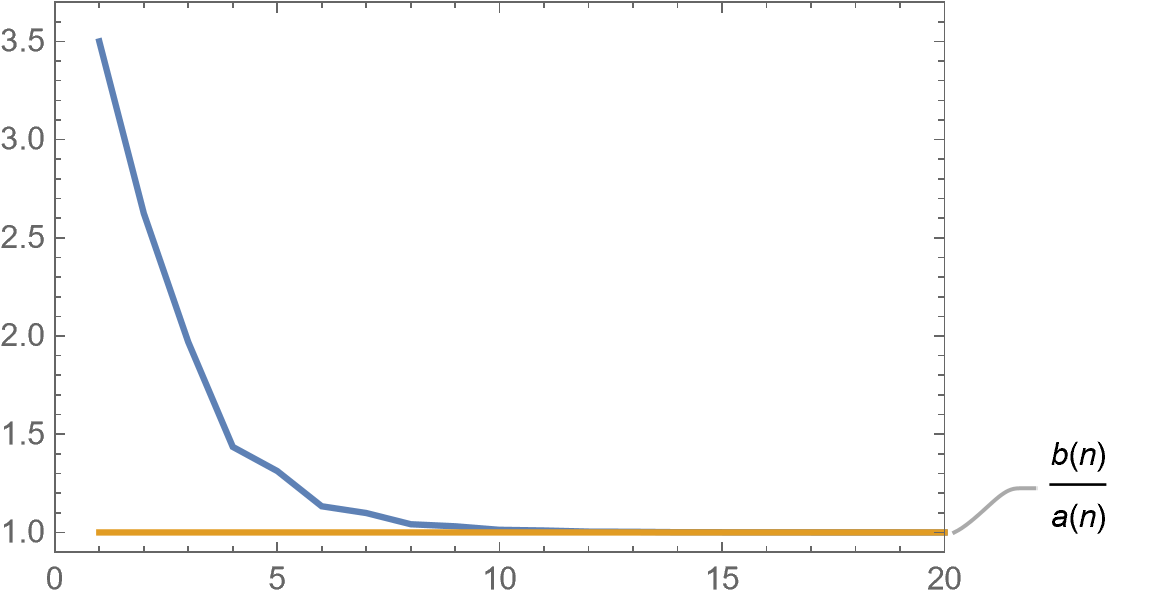}
    \caption{The ratio $b(n)/a(n)$ for the four-dimensional irreducible $k\text{SL}(2,7)$-module in characteristic 7. Here {the ratio of convergence is} $\lvert \lambda^{\mathrm{sec}}\rvert/\dim V \approx 0.56.$}
    \label{fig:sl2,7 4d ratio}
\end{figure}
\begin{Example} Let $G=\text{SL}(2,8)$ and $\text{char\ } k = 2$. For the two-dimensional faithful irreducible $kG$-module $V$ we have $a(n)=3/56 \cdot 2^n$
and $\lambda^{\mathrm{sec}}=\omega^5-\omega^2-\omega$, where $\omega=\exp(2\pi i/18)$. We plot the fusion graph and action matrix in Figure \ref{fig:sl(2,8)}. In this case {the ratio of convergence coincides with} $\lvert \lambda^{\mathrm{sec}}\rvert /\dim V \approx 0.94$, which explains the very slow rate of geometric convergence illustrated in Figure \ref{fig:SL28 convergence}.     
\end{Example}
 \begin{figure}[H]
    \centering
    \begin{minipage}{0.40\textwidth}
    \centering
\includegraphics[width=\textwidth]{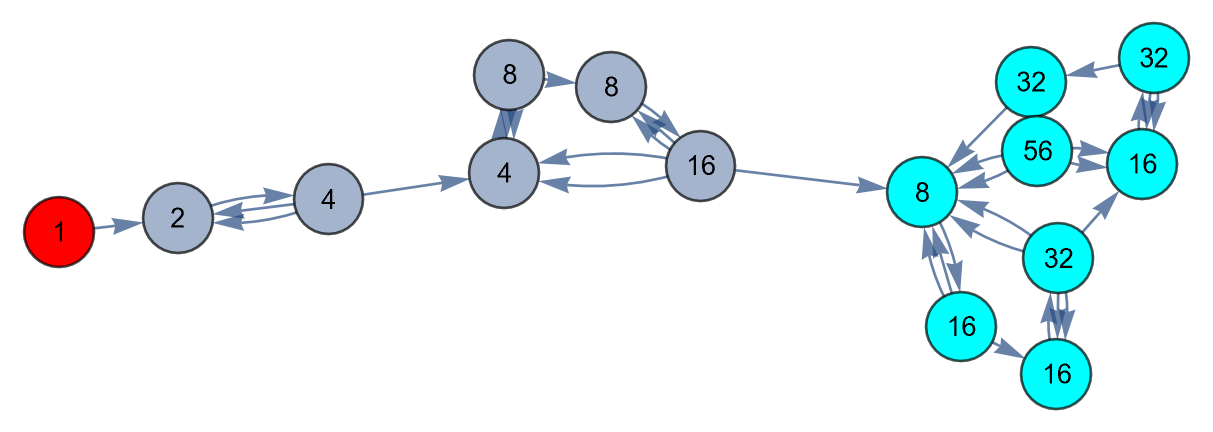} 
    \end{minipage}
    \begin{minipage}{0.4\textwidth}
     \centering   
     $$\scalemath{0.6}{\begin{pmatrix}
0 & 0 & 0 & 0 & 0 & 0 & 0 & 0 & 0 & 0 & 0 & 0 & 0 & 0 & 0 \\
1 & 0 & 2 & 0 & 0 & 0 & 0 & 0 & 0 & 0 & 0 & 0 & 0 & 0 & 0 \\
0 & 1 & 0 & 0 & 0 & 0 & 0 & 0 & 0 & 0 & 0 & 0 & 0 & 0 & 0 \\
0 & 0 & 1 & 0 & 2 & 0 & 2 & 0 & 0 & 0 & 0 & 0 & 0 & 0 & 0 \\
0 & 0 & 0 & 1 & 0 & 0 & 0 & 0 & 0 & 0 & 0 & 0 & 0 & 0 & 0 \\
0 & 0 & 0 & 0 & 1 & 0 & 2 & 0 & 0 & 0 & 0 & 0 & 0 & 0 & 0 \\
0 & 0 & 0 & 0 & 0 & 1 & 0 & 0 & 0 & 0 & 0 & 0 & 0 & 0 & 0 \\
0 & 0 & 0 & 0 & 0 & 0 & 1 & 0 & 2 & 0 & 2 & 0 & 0 & 1 & 2 \\
0 & 0 & 0 & 0 & 0 & 0 & 0 & 1 & 0 & 0 & 0 & 0 & 0 & 0 & 0 \\
0 & 0 & 0 & 0 & 0 & 0 & 0 & 0 & 1 & 0 & 2 & 0 & 0 & 0 & 0 \\
0 & 0 & 0 & 0 & 0 & 0 & 0 & 0 & 0 & 1 & 0 & 0 & 0 & 0 & 0 \\
0 & 0 & 0 & 0 & 0 & 0 & 0 & 0 & 0 & 0 & 1 & 0 & 2 & 0 & 2 \\
0 & 0 & 0 & 0 & 0 & 0 & 0 & 0 & 0 & 0 & 0 & 1 & 0 & 0 & 0 \\
0 & 0 & 0 & 0 & 0 & 0 & 0 & 0 & 0 & 0 & 0 & 0 & 1 & 0 & 2 \\
0 & 0 & 0 & 0 & 0 & 0 & 0 & 0 & 0 & 0 & 0 & 0 & 0 & 1 & 0 \\
\end{pmatrix}}
$$
    \end{minipage}
    \caption{The fusion graph and action matrix for the two-dimensional irreducible $k\text{SL}(2,8)$-module in characteristic 2. The trivial and projective modules are colored red and cyan respectively.} 
    \label{fig:sl(2,8)}
\end{figure}
\begin{figure}[H]
    \centering
    \includegraphics[width=0.5\linewidth]{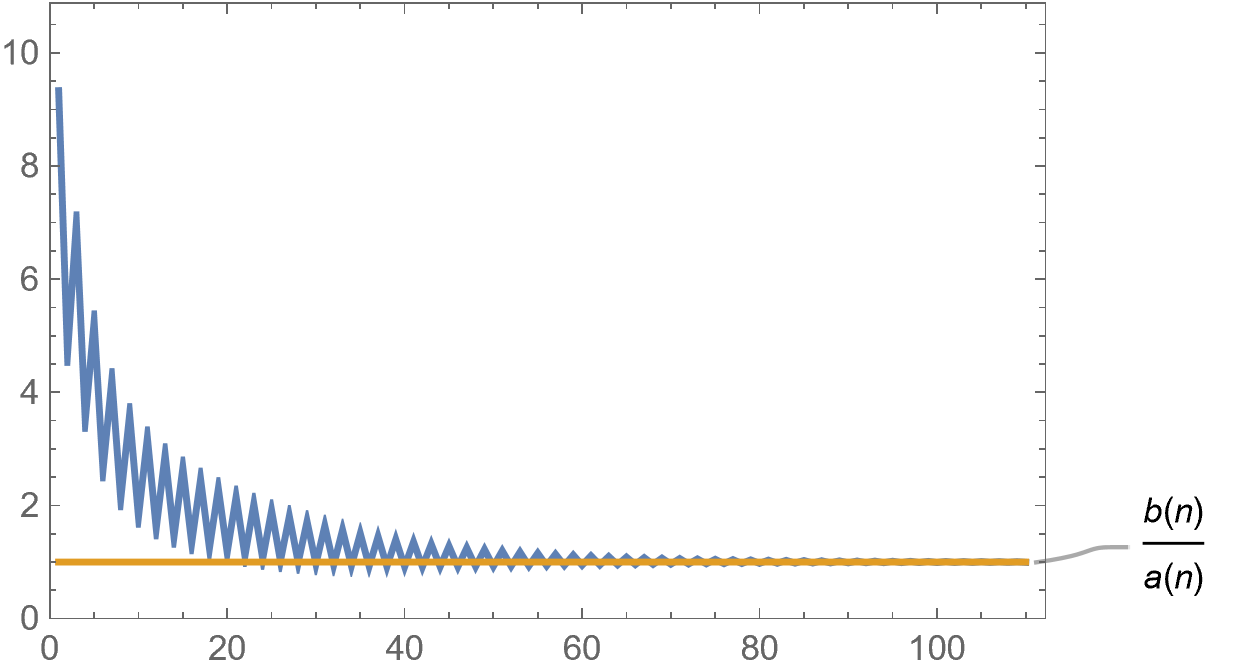}
    \caption{The ratio $b(n)/a(n)$ for the two-dimensional irreducible  $k\text{SL}(2,8)$-module in characteristic 2.}
    \label{fig:SL28 convergence}
\end{figure}}
{We saw in Example \ref{irr counterexample} that in the modular case the character table does not, in general, contain information about the second largest eigenvalue. We now turn to two examples where $a(n)$ is easy, because $G$ is a $p$-group, and determine the second largest eigenvalues (and hence obtain bounds on the rate of convergence and variance).}

\subsection{Cyclic groups}
Let $G=C_p$ be a cyclic group of prime order{$p >2$}, and let $k$ be an algebraically closed field of characteristic $p$. By Corollary \ref{p group} we know the asymptotic formula for a faithful $kG$-module {is} $a(n)=1/p \cdot (\dim V)^n$, so it remains to determine the rate of convergence and variance. We will do this by computing $\lvert \lambda^{\mathrm{sec}}\rvert$ for the action matrix associated with each indecomposable $kG$-module. For $G=C_p$, denote by $V_1,\dots, V_p$ the $p$ indecomposable $kG$-modules of dimension $1,\dots, p$. We have the following formula for tensor product decomposition:
\begin{Lemma}[\textup{\cite[Theorem 1]{renaud1979decomposition}}]\label{renaud}
 Let $V_l$, $1\le l\le p$ be as above. For $1\le r\le s \le p$, we have $$V_r \otimes V_s = \bigoplus_{i=1}^c V_{s-r+2i-1}\oplus V_p^{{\oplus (r-c)}}, \text{\ where \ }
c= \begin{cases}
    r & r+s\le p,\\
    p-s & r+s\ge p.
\end{cases}$$
\end{Lemma}

The formula implies that when $l$ is odd, tensor powers of $V_{l}$ contain only summands $V_j$ where $j$ is odd. When $l$ is even, tensor powers of $V_l$ contain both odd and even $V_j$'s as summands. We give the fusion graph and action matrix of $V_3$ (as a $kC_5$-module) in Figure \ref{fig:V3}. Since $3$ is odd, the even and odd dimensional modules form two connected components. (Here we use $m$ edges to represent an edge of weight $m$). For $1\le l \le p$, we denote by $M_l$ the action matrix of $V_l$. We assume all action matrices are with respect to the basis $(V_1,\dots,V_p).$

\begin{figure}[h]
    \centering
    \begin{minipage}{0.40\textwidth}
        \centering
\includegraphics[width=1\textwidth]{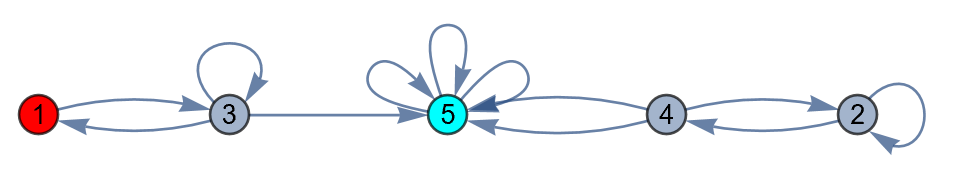} 
    \end{minipage}
    \begin{minipage}{0.4\textwidth}
        \centering
        \small
     $$\begin{pmatrix}
     0 & 0 & 1 & 0 & 0\\
     0 & 1 & 0 & 1 & 0\\
     1 & 0 & 1 & 0 & 0\\
     0 & 1 & 0 & 0 & 0\\
     0 & 0 & 1 & 2 & 3
    \end{pmatrix}$$   
    \end{minipage}
    \caption{The fusion graph and action matrix of $V_3$ (as a $kC_5$-module). The vertices are labeled with their dimensions. The trivial and projective modules are colored red and cyan respectively.}
    \label{fig:V3}
\end{figure}

\begin{Theorem}[Cyclic groups] \label{cyclic}
Suppose $p$ is an odd prime with $p=2n+1$. Let $\lambda^{\mathrm{sec}}_l$ denote the modulus of the second largest eigenvalue of $M_l$. We have that for $1\le l \le n$,
 $$\lvert \lambda^{\mathrm{sec}}_l\rvert =\lvert \lambda^{\mathrm{sec}}_{p-l}\rvert =
 \phi_n(l):=\frac{\sin({l\pi}/(2n+1))}{\sin({\pi}/(2n+1))}$$
 where out of convenience we define $\lambda_{1}^{\mathrm{sec}}$ to be 1.
  
\end{Theorem}
We note that $\phi_n(l)$ is the length of the $l$-th shortest diagonal in a regular $(2n+1)$-gon with unit side length, or equivalently the modulus of the sum of $l$ consecutive $(2n+1)$-th roots of unity: \cite{anderson20fibonacci} calls these the \textit{golden numbers}.  We defer the proof of Theorem \ref{cyclic} to the end of the section.

We plot the ratio ${b(n)}/{a(n)}$ for the nontrivial indecomposable $kC_5$-modules in Figure \ref{fig:c5 graph}. The plot for $V_5$ is omitted since in this case $a(n)=b(n)$. For $V_2,V_3,V_4$, the modulus of the second largest eigenvalue are given by $\phi,\phi$ and 1 respectively, where $\phi=\phi_2(2)=\phi_3(2)\approx 1.618$ is the golden ratio. Thus, we have {${\lvert \lambda_2^{\mathrm{sec}}\rvert}/{\dim V_2}\approx 0.809,{\lvert \lambda_3^{\mathrm{sec}}\rvert}/{\dim V_3}\approx 0.539,{\lvert \lambda_4^{\mathrm{sec}}\rvert}/{\dim V_4}=0.25$}, which explains the rate of convergence shown in Figure \ref{fig:c5 graph}.

\begin{figure}[H]
    \centering
    \begin{minipage}{0.40\textwidth}
    \centering
\includegraphics[scale = 0.58]{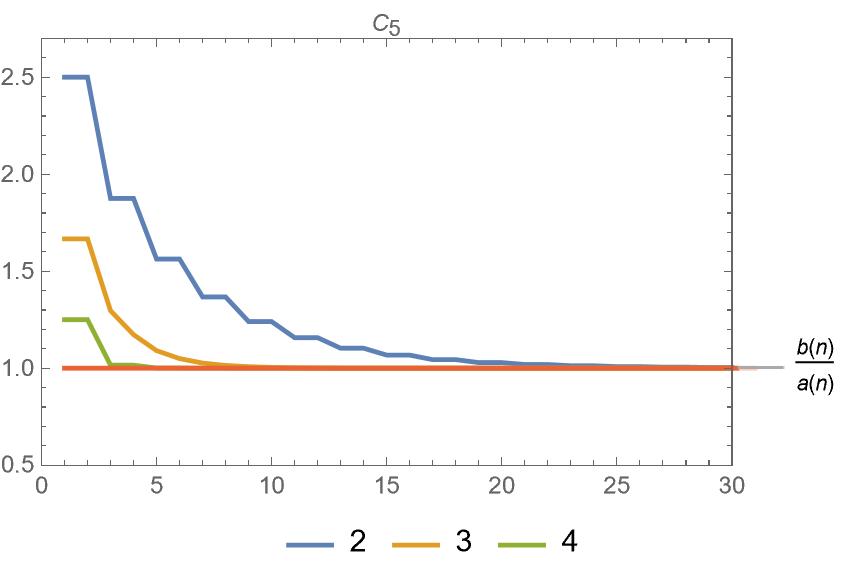}
    \caption{$p=5$} 
    \label{fig:c5 graph}   
    \end{minipage}
       \begin{minipage}{0.40\textwidth}
    \centering
\includegraphics[scale=0.38]{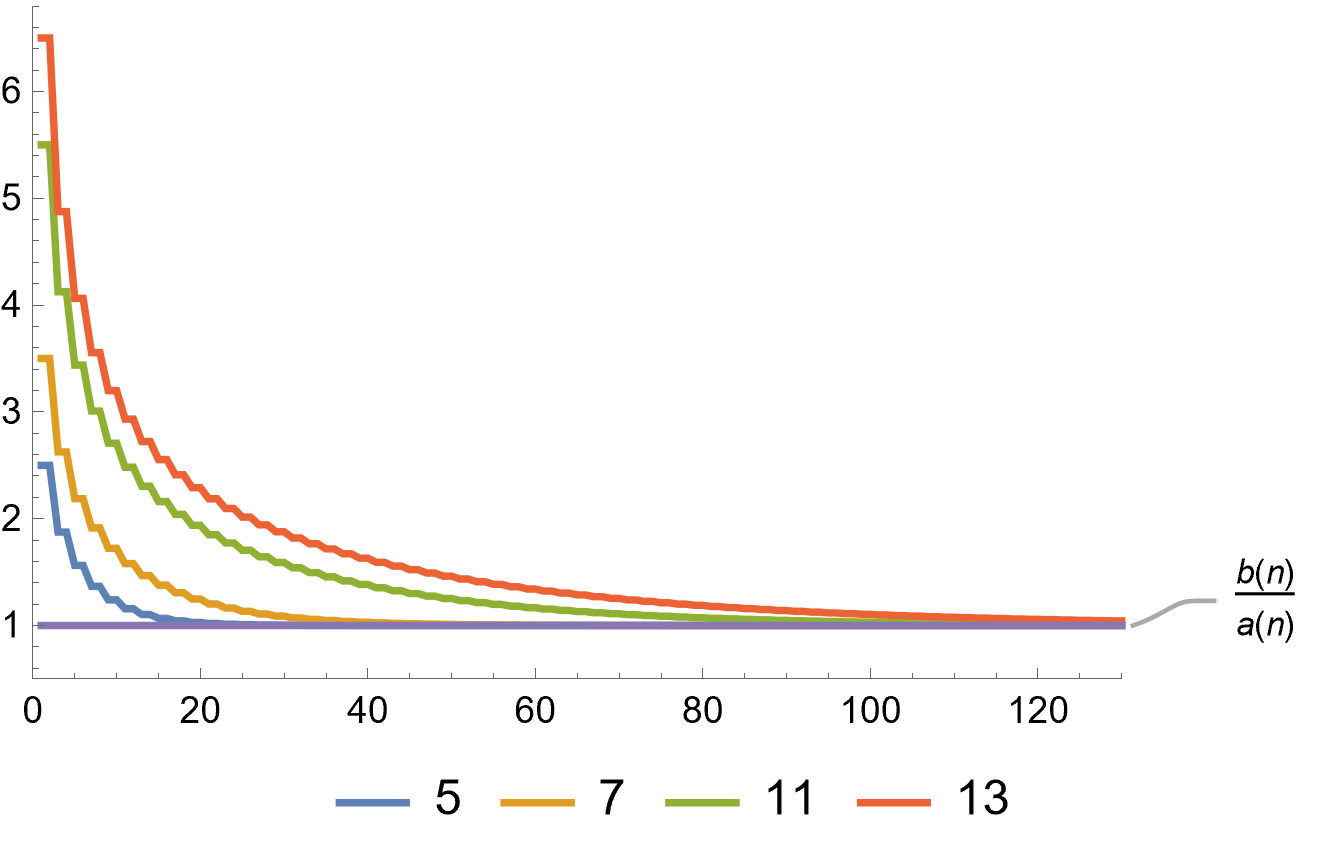}
    \caption{$l=2$, varying $p$.}
    \label{fig:slowing conv}
    \end{minipage}
\end{figure}

The next result easily follows from Theorem \ref{cyclic}.
\begin{Proposition}
Let $p$ be an odd prime such that $p=2n+1$. Let $\lambda_l^{\mathrm{sec}}$ be as before. For $1\le l \le n$, we have $\lim\limits_{p \to \infty}\lvert \lambda_l^{\mathrm{sec}}\rvert = \lim\limits_{p \to \infty}\lvert \lambda_{p-l}^{\mathrm{sec}}\rvert \to l.$ In particular, {we have} $$ \frac{\lvert \lambda_l^{\mathrm{sec}}\rvert }{l} \xrightarrow{ p \to \infty } 1.$$     
\end{Proposition}

We illustrate below in Table \ref{Tab:ratio of convergence} how for fixed $l$, the ratio ($\lvert \lambda_l^{\mathrm{sec}}\rvert /l$) varies as $p$ increases. The numbers are rounded to four decimal places. We also plot the ratio $b(n)/a(n)$ for the two-dimensional irreducible $kC_p$ module $V_2$ as we increase $p$ in Figure \ref{fig:slowing conv}. Already for $p=13$ the convergence is very slow (in this case $\lvert \lambda_2^{\mathrm{sec}}\rvert /2 \approx 0.9709$).

\begin{table}[H]
\begin{tabular}{c|c c c c c c c c c c c}
   $p$  & $5$ & $7$ & $11$ & $13$ & $17$ & $19$ & $23$ & 29\\
     \hline
   $\lvert \lambda_2^{\mathrm{sec}}\rvert /2$  & $0.8090$ & $0.9010$ & $0.9595$ & $0.9709$ & $0.9830$ & $0.9864$ & $0.9907$ & $0.9941$\\
   $\lvert \lambda_3^{\mathrm{sec}}\rvert /3$ & $0.5393$ & $0.7490$ & $0.8941$ & $0.9236$ & $0.9550$ & $0.9639$ & $0.9753$ & $0.9844$ \\
   $\lvert \lambda_4^{\mathrm{sec}}\rvert /4 $ & $0.25$ & $0.5617$ & $0.8072$ & $0.8597$ & $0.9166$ & $0.9329$ & $0.9539$ & $0.9709$\\
   $\lvert \lambda_5^{\mathrm{sec}}\rvert /5$ & $0$ & $0.3604$ & $0.7027$ & $0.7814$ & $0.8686$ & $0.8940$ & $0.9269$ & $0.9537$\\
 \end{tabular}
 \caption{{The ratio $\lvert \lambda_w^{\text{sec}}\rvert /w$ for the $w$-dimensional indecomposable $kC_p$ module.}}
 \label{Tab:ratio of convergence}

\end{table}

We now prove Theorem \ref{cyclic}.
\begin{proof}[Proof of Theorem \ref{cyclic}]

We need the following lemma: 
\begin{Lemma} \label{eigenvalue cyclic}
For $2 \le l \le p-1$, set $l=2k$ if $l$ is even, and $l=p-2k$ if odd, and let $\omega=\exp(2\pi i/{p})$. Then the eigenvalues of $M_l$ apart from $\dim V_l = l$ are $$ \pm \sum_{j=1}^{2k}\big(\omega^{\frac{p-2k+2j-1}{2}}\big)^m, 1\le m \le (p-1)/2,$$ if $l$ is even. If $l$ is odd, only the negative sums are eigenvalues, with multiplicity 2.      
\end{Lemma}  
\begin{proof}
Since the first column of $M_l$ has a 1 in the $l$-th entry and nowhere else, it suffices to show that the left eigenvectors other than $[1,\dots, p]$ come in two series which are independent of $l$:   
\begin{enumerate}
    \item $\x^m$ for $1\le m \le (p-1)/2$, given by $\x^m_1=1$, $\x^m_p=0$,  $\x^m_{2k} = \sum_{j=1}^{2k}\left(\omega^{\frac{p-2k+2j-1}{2}}\right)^m$ and $\x^m_{p-2k}=-\x^m_{2k}$, where by $\x^m_j$ we denote the $j$-th entry of $\x^m$, and 
    \item $\y^m$ for $1\le m \le (p-1)/2$, with $\y^m_1=1,\y^m_p=0$, $\y^m_{2k} = \y^m_{p-2k} = \x^m_{p-2k}$.\end{enumerate}  
That these are (all the) left eigenvectors can be verified by tedious but straightforward calculation: one shows that each $\x^m$ and $\y^m$ satisfies the formula given in Lemma \ref{renaud} in the sense that for each $m$ and
$1\le r\le s \le p$, $$\x^m_r \cdot \x^m_s = \sum_{{t}=1}^c \x^m_{s-r+2{t}-1}+(r-c)\x^m_p \text{\ , where } c= \begin{cases}
    r & r+s\le p,\\
    p-s & r+s\ge p 
\end{cases}$$
and similarly for each $\y^m$.
\end{proof}
Lemma \ref{eigenvalue cyclic} implies that eigenvalues for $V_l$ (other than the dimension) are all sums of $l$ different roots of unity. It is easy to see that the the modulus of the sum of $l$ different $p$-th roots of unity attains its maximum when these roots are consecutive (when thought of as evenly spaced points on the unit circle), which now implies the statement of Theorem \ref{cyclic}.  
\end{proof}
\subsection{Klein four group}
Let $G=V_4$ be the Klein four group, and take $k$ to be an algebraically closed field of characteristic 2. Again Corollary \ref{p group} applies, so we focus on the second largest eigenvalue. We have that $kG$ is of tame representation type, and the indecomposable $kG$-modules are completely classified, see \textit{e.g.} \cite[Section 11.5]{webb2016course}. The growth problem for the special case $M_3$ has been studied in \cite[Example 4.14; Example 5.16]{lacabanne2024asymptotics}. 

\begin{Theorem}[Klein four group]
Let $G=V_4$. We summarise our results for the faithful indecomposable $kG$-modules below:
\begin{center}
    \noindent
\begin{tabular}{c|c|c|c|c}
    Type & Dimension & Diagram & Action matrix &  $\lambda^{\mathrm{sec}}$ \\
     \hline
     $M_{2m+1}$ (and its dual) & $2m + 1$ &
     $\left(\vcenter{\hbox{
         \adjustbox{scale=0.5}{%
         \begin{modularrep}
                    & \arrow[dl] {} \arrow[dr] & & \arrow[dl] {} \\
                    {} & & {}
                \end{modularrep}
                $\cdots$
                \begin{modularrep}
                    {} \arrow[dr] \\
                    & {}
                \end{modularrep}}
            }}\right)$ & \tiny $\begin{pmatrix}
      2m+1 & 0 & m^2 & 2m^2 & 3m^2 & \cdots \\
      0 & 0 & 0 & 0 & 0 &  \cdots \\
     0 & 1 & 0 & 0 & 0 &   \cdots\\
        0 & 0 & 1 & 0 & 0  & \cdots\\
        0 & 0 & 0 & 1 & 0  & \cdots\\
        \vdots  & \vdots & \vdots & \vdots & \vdots & \ddots
    \end{pmatrix}$ 
    & $0$\\
    \hline
    \begin{tabular}{@{}l@{}}
        $E_{f,m}$ (and \\ $E_{0,m},E_{\infty,m}, m\neq 1$)
    \end{tabular} &
    $2ml$ &
    $\left(\vcenter{\hbox{
        \adjustbox{scale=0.5}{%
                \begin{modularrep}
                    & \arrow[dl] {} \arrow[dr] & & \arrow[dl] {} \\
                    {} & & {}
                \end{modularrep}
                $\cdots$
                \begin{modularrep}
                    & \arrow[dl] {} \arrow[dr] \\
                    {} & & |[fill=none]|
                \end{modularrep}}
            }}\right)$ & \tiny $\begin{pmatrix}
      0 & 0 & 0\\
      1 & 2 & 0\\
      0 & ml(ml-1)&2ml
     \end{pmatrix}$ & $2$ \\
     \hline
     $E_{f,1}, l \neq 1 $ & $2l$ & same as above & \tiny $\begin{pmatrix}
       0 & 0 & 0 & 0\\
       0 & 2l & l(l-1) & l(2l-1)\\
       1 & 0 & 0 & 2\\
       0 & 0 & 1 & 0
     \end{pmatrix}$ & $\pm \sqrt{2}$ \\
     \hline
     $kV_4$ & $4$ & $\left(\vcenter{\hbox{
         \adjustbox{scale=0.5}{%
                \begin{modularrep}
                    & \arrow[dl] {} \arrow[dr] & \\
                    {} \arrow[rd] & & \arrow[dl] {} \\
                    & {}
                \end{modularrep}}
            }}\right)$ & \tiny $\begin{pmatrix}
      0 & 0\\
      1 & 4
     \end{pmatrix}$& $0$
    \end{tabular}
\end{center}   
The infinite matrix is with respect to the basis
$(kV_4,M_1,M_{2m+1},M_{4m+1},M_{6m+1},\hdots)$. The other three matrices are with respect to $(M_1, E_{f,{m}},kV_4),(M_1,kV_4,E_{f,1},E_{f,2})$ and $(M_1,kV_4)$ respectively, where we allow $f$ to be replaced by $0$ or $\infty$.

For the cases with finite matrices, we also have the exact growth rate $b(n)$:
\begin{center}
    \begin{tabular}{c|c}
       Type  & $b(n)$  \\
       \hline
       $E_{f,m},E_{0,m},E_{\infty,m},m\neq 1$  & $\frac{1}{4}\cdot (2ml)^n+\frac{2-ml}{4}\cdot 2^n$\\
       \hline
       $E_{f,1},l\neq 1$ & $\frac{1}{4}\cdot (2l)^n+\frac{l(\sqrt{2}-2)-2\sqrt{2}+2}{8}\cdot (-\sqrt{2})^n+\frac{l(-2-\sqrt{2})+2+2\sqrt{2}}{8}\cdot (\sqrt{2})^n$\\
       \hline
       $kV_4$ & $4^{n-1}=a(n).$
    \end{tabular}
\end{center}

\end{Theorem}
We have used the diagrammatic notation used in \textit{e.g.} \cite[Section 11.5]{webb2016course}. To briefly explain, $m$ is a positive integer, and $f \in k[x]$ is an irreducible polynomial of degree $l$. In the diagrams, each node $v_i$ represents a basis element of the underlying vector space. Let $G=\langle a\rangle \times \langle b\rangle$, then a southwest (resp. southeast) arrow from $v_i$ to $v_j$ means that $a-1$ (resp. $b-1$) acts by mapping $v_i$ to $v_j$. To understand the diagram for the family ${E_{f,m}}$, label the $ml$ vertices in the first row by $u_0,\hdots,u_{ml-1}$, and the vertices in the second row by $v_0,\hdots,v_{ml-1}$. Let $(f(x))^m=x^{ml}+a_{ml-1}x^{ml-1}+\hdots+a_0$. The arrow extending from $u_{ml-1}$ is then to be interpreted as follows: it represents that
$$(b-1)(u_{ml-1})=-a_{ml-1}v_{ml-1}-\hdots-a_1v_1-a_0v_0.$$

We set $l=1$ for $E_{0,m}$ or $E_{\infty,m}$. We have omitted the cases where the module is not faithful, namely $E_{0,1},E_{\infty,1}$, and $E_{f,1}$ where $f$ has degree $l=1$. 

\begin{proof}
The matrices above are obtained using the tensor product formulas in \cite[Section 3]{conlon1965certain}. (In the notation there, $E_{f,m}$ is called $C_{n}(\pi)$, and $M_{2n+1}$ and its dual are called $A_n$ and $B_n$.) Everything else follows from straightforward calculation.    
\end{proof}

For each of the infinite families, $\lambda^{\mathrm{sec}}$ is independent of the dimension of the $kG$-module. This implies, in particular, that ${\lvert \lambda^{\mathrm{sec}}\rvert}/{\dim V}\to 0$ as $\dim V \to \infty$, and so for each of the families with finite graph the {ratio of convergence can be made arbitrarily small} by taking a $kG$-module with large enough dimension. In fact, as $2,\sqrt{2}$ are rather small numbers, the convergence is already very fast when the dimension is 6, see Figure \ref{fig:Klein finite}.

\begin{figure}[h!]
    \centering
    \includegraphics[width=0.5\linewidth]{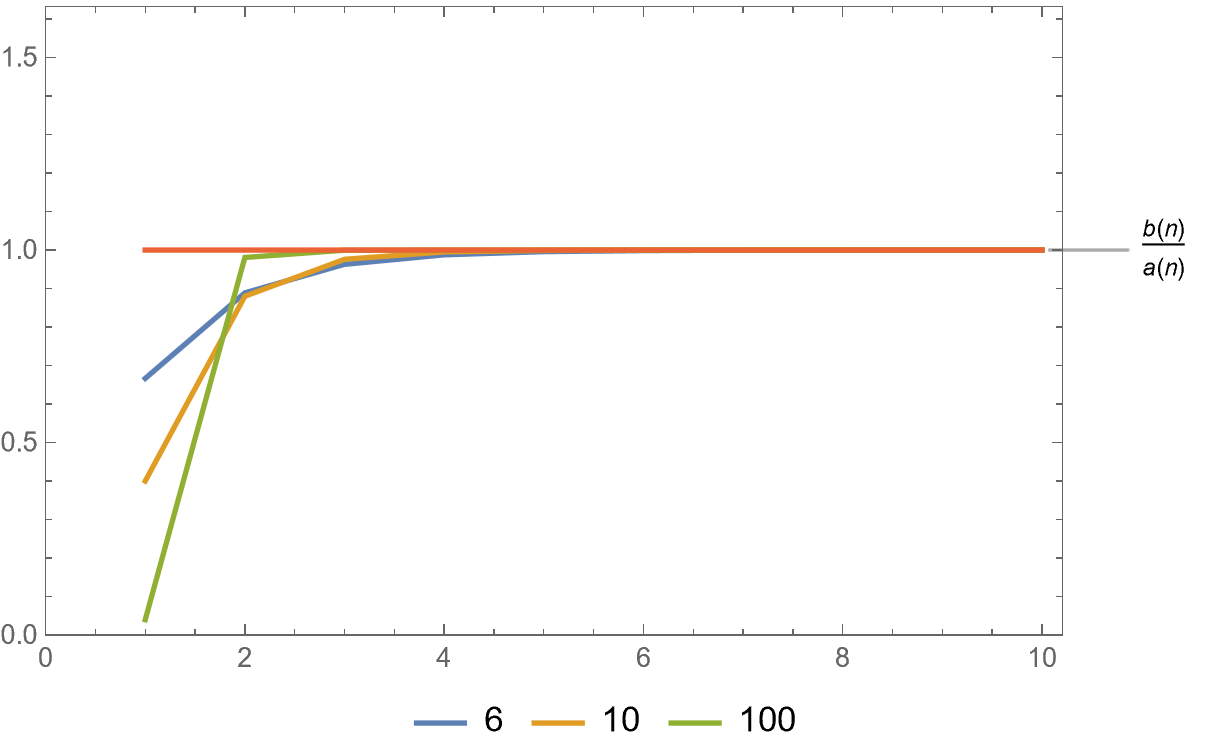}
    \caption{The ratio $b(n)/a(n)$ for $E_{f, m},m\neq 1$ with dimensions $2ml=6,10,100$. The ratio of convergence is ${\lvert \lambda^{\mathrm{sec}}}\rvert/{\dim E_{f,m}}={2}/{2ml}$ which goes to 0 as $2ml\to \infty$. We also note that when $2ml=4$ we have $b(n)=a(n)$.}
    \label{fig:Klein finite}
\end{figure}

\begin{Remark} \label{infinite sec}
Recall that in the infinite case, a subexponential correction factor is introduced so that $\lvert b(n)-a(n)\rvert \in \mathcal{O}(n^{d}+(\lambda^{\mathrm{sec}})^n)$ for some $d>0$. Thus, even though in the $M_{2m+1}$ case we have $\lambda^{\mathrm{sec}}=0$, we do not have a closed form for $b(n)$. See also \cite[Example 5.16]{lacabanne2024asymptotics} for an illustration.    
\end{Remark}

\section*{Appendix} \label{appendix}
\renewcommand{\thesubsection}{\Alph{subsection}}
\setcounter{subsection}{0}

This appendix discusses how we have used Magma to compute various formulas. We display the code in blue and the output in red. The Magma code below can all be directly pasted into, for example, the online Magma calculator here: \url{http://magma.maths.usyd.edu.au/calc/}.

\subsection{Computing the exact formula using Magma}
The following code asks Magma to compute $b(n)$ for a $\C G$-module $V$ with character $\chi$, by summing over $\langle \chi^n, \chi_j \rangle$ where $\{\chi_1,\dots,\chi_N\}$ is a complete set of irreducible characters for $G$. This gives the values which we hope to find a formula for. 


\bigskip

\begin{small}
{ 
\color{blue}%
\begin{verbatim}
> G:=DihedralGroup(6); //Take any group
> X:=CharacterTable(G);
> V:=X[5]; //Fix a representation
> m:=10; //Maximal tensor power 
> outlist:=[];

> for n in [1..m] do
> out:=0;
> for j in [1..#X] do
> out+:=InnerProduct(V^n,X[j]); //Sum over the inner products
> end for;
> outlist:=Append(outlist,out);
> end for;
> outlist;
\end{verbatim}
}
{\color{red} 
\begin{verbatim}
[1, 3, 5, 11, 21, 43, 85, 171, 341, 683]
\end{verbatim}
}
\end{small}

\bigskip

Next, we compute $b(n)$ using Theorem \ref{exact formula}. First we calculate the column sums, then we plug them and the values of $\chi$ into the formula.

\bigskip

\begin{small}
{\color{blue}
\begin{verbatim}
>col:=[];
>for j in [1..#X] do
>value:=0;
>for i in [1..#X] do
>value+:=X[i][j]; //Sum over all character values in the j-th column
>end for;
>col:=Append(col,ClassesData(G)[j][2]*value); //Scale by conjugacy class size
>end for;
  
>result:=[];
>for n in [1..m] do
>entry:=0;
>for j in [1..#X] do
>entry:= entry + col[j]*(V[j])^n;
>end for;
>result:=Append(result,(1/Order(G))*entry); //Plug into the formula
>end for;
>result; 
\end{verbatim}
}
{\color{red}
\begin{verbatim}
[1, 3, 5, 11, 21, 43, 85, 171, 341, 683]
\end{verbatim}
}
\end{small}

This verifies Theorem \ref{exact formula} in this particular case.

\subsection{Computing the action matrix and asymptotic formula using Magma} \label{action magma}
Given a $kG$-module $V$, we can also ask Magma to compute (a finite submatrix of) the corresponding action matrix $M(V)$.  In the nonsemisimple case $M(V)$ could be infinite, so we will specify a cutoff value $k$ and ask Magma to compute the submatrix $M_k(V)$. First we need the vertex set of the subgraph $\Gamma_k$, consisting of all pairwise nonisomorphic indecomposable summands of $V^{\otimes l}, 0\le l \le k$, which gives a basis for $M(V)$. We used the code below to computing the matrix, the eigenvalues, and the asymptotic growth rate for Example \ref{irr counterexample}.

\bigskip

\begin{small}
{
\color{blue}
\begin{verbatim}
>G:=SL(2,7);
>k:=GF(7); //Take a large enough field of characteristic 7
>Irr:=IrreducibleModules(G,k); //All irreducible kG-modules 
>M:=Irr[4]; //Take the 4 dimensional module

>X:=[Irr[1]]; //The vertex set, initially containing only the trivial module
>nn:=0;
>max:=12; //Choose a cutoff value
>while nn lt #X and nn le max do
>nn+:=1;
>M2:=TensorProduct(X[nn],M);
>XM2:=IndecomposableSummands(M2); //Tensor with existing vertices to get new modules
>for j in [1..#XM2] do
>new:=1;
>for i in [1..#X] do
>if(IsIsomorphic(XM2[j],X[i])) then //Check if the vertices are already in the list
>new:=0;
>end if;
>end for;
>if(new eq 1) then X:=Append(X,XM2[j]); //Add the new modules to the list 
>end if;
>end for;
>end while;

>fus:=[[0 : i in [1..#X]] : j in [1..#X]];
>for i in [1..#X] do
>for j in [1..#X] do
>Z:=IndecomposableSummands(TensorProduct(X[i],M));
>z:=0;
>for k in [1..#Z] do
>if(IsIsomorphic(X[j],Z[k])) then z:=z+1; end if;
>end for;
>fus[j][i]:=z; //Create the action matrix
>end for;
>end for;
>Ma:=Matrix(CyclotomicField(56),fus); //Use a cyclotomic field to get all eigenvalues
>Ma;
\end{verbatim}
}
{\color{red}
\begin{verbatim}

[0 1 0 0 0 0 0 0 0 0 0 0 0]
[1 0 1 1 0 0 0 0 0 0 0 0 0]
[0 1 0 0 0 1 1 0 0 0 0 0 0]
[0 1 0 0 0 1 0 0 0 0 0 0 0]
[0 1 0 0 0 0 1 2 3 0 0 0 2]
[0 0 1 1 0 0 0 0 0 0 0 0 0]
[0 0 1 0 0 0 0 0 0 0 0 0 0]
[0 0 0 1 1 0 0 0 0 2 2 1 0]
[0 0 0 0 1 0 0 0 0 1 1 1 0]
[0 0 0 0 0 0 1 2 1 0 0 0 2]
[0 0 0 0 0 0 0 1 1 0 0 0 1]
[0 0 0 0 0 0 0 0 1 0 0 0 0]
[0 0 0 0 0 0 0 0 0 1 1 0 0]
\end{verbatim}}
{\color{blue}
\begin{verbatim}
    
>Eigenvalues(Ma);
\end{verbatim}
}
{\color{red}
\begin{verbatim}
{
    <zeta_56^20 - zeta_56^16 + zeta_56^12 - zeta_56^8, 1>,
    <-1, 1>,
    <-zeta_56^20 + zeta_56^8 + 1, 1>,
    <-zeta_56^20 + zeta_56^16 - zeta_56^12 + zeta_56^8, 1>,
    <-4, 1>,
    <1, 1>,
    <-zeta_56^16 + zeta_56^12 - 1, 1>,
    <zeta_56^20 - zeta_56^8 - 1, 1>,
    <zeta_56^16 - zeta_56^12 + 1, 1>,
    <0, 3>,
    <4, 1>
}
\end{verbatim}   
}
\end{small}
(Magma represents the primitive root of unity $\exp(2\pi i/m)$ by \verb|zeta_m|, so the non-integral eigenvalues are all combinations of $56$-th roots of unity.)
\bigskip

In this case the action matrix is finite by \cite[Theorem B]{craven2013tensor}, and we have obtained it. In general, we would only have a submatrix which, when $k$ is large, approximates the complete matrix. From the above we can see that there are two eigenvalues with modulus equal to $\dim V=4$, so there will be two terms in the formula for $a(n)$ by Theorem \ref{genformula}. We use left and right eigenvectors to work out the coefficients in front of these terms. We note that by default Magma computes the left eigenspaces, so to get the right eigenvectors we need to first transpose the matrix. 

\bigskip

\begin{small}
{\color{blue}
\begin{verbatim}
>w0:=Eigenspace(Ma,Dimension(M)).1; //Compute the left eigenvector w_0^{T} 
>vv0:=Eigenspace(Transpose(Ma),Dimension(M)).1; 
>v0:=vv0/ScalarProduct(w0,vv0); //Normalised right vector so that w_0^{T}v_0=1
>&+[(Matrix(Rationals(),#X,1,ElementToSequence(v0))*Matrix(Rationals(),1,#X,
ElementToSequence(w0)))[i][1] : i in [1..#X]]; //Sum over first column of v_0w_0^{T}   
\end{verbatim}
}
{\color{red}
\begin{verbatim}
1/12
\end{verbatim}
}
\end{small}

\bigskip

The constant $1/12$ is $ \sum_{W \in \text{Irr}_k(G)}\dim W/\lvert G \rvert$. We can check this:

\bigskip

\begin{small}
{\color{blue}
\begin{verbatim}
> &+[Dimension(Irr[i]):i in [1..#Irr]]/Order(G)
\end{verbatim}
}
{\color{red}
\begin{verbatim}    
1/12
\end{verbatim} }   
\end{small}

\bigskip

Similarly, we can calculate $v_1w_1^T[1]$ to be $1/84$, so we have
\begin{equation} \label{eqn: magma example}
a(n)=\Big(\frac{1}{12}+\frac{1}{84}(-1)^n\Big)\cdot 4^n.    
\end{equation}

Next we check that $b(n)/a(n)$ converges to 1. Once we have the action matrix, it is easy to compute $b(n)$: we only have to sum over the first column of $\big(M(V)\big)^n$. 

\begin{small}
{\color{blue}
\begin{verbatim}
>m:=15;
>exact:=[];
>for n in [1..m] do
>exact:=Append(exact,&+[(Ma^n)[i][1] : i in [1..#X]]); 
>end for; 
>exact;
\end{verbatim}
}
{\color{red} \begin{verbatim}
[1, 4, 9, 35, 96, 442, 1286, 6502, 19309, 101178, 302544, 1604461, 4808389, 25598759,
76771067]
   
\end{verbatim} 
}   
\end{small}

To compute $a(n)$, we simply use \eqref{eqn: magma example}.

\begin{small}

{\color{blue}\begin{verbatim}
>asym:=[];
>for n in [1..m] do
>asym:=Append(asym, RealField(9)!(1/12+(-1)^n/84)*4^n);
>end for;
>asym;
\end{verbatim}
}
{\color{red} \begin{verbatim}
[0.285714286, 1.52380952, 4.57142857, 24.3809524, 73.1428572, 390.095238,
1170.28572, 6241.52381, 18724.5714, 99864.3810, 299593.143, 1597830.10,
4793490.29, 25565281.5, 76695844.6]
\end{verbatim}
}

{\color{blue} \begin{verbatim}
>ratio:=[];
>for n in [1..m] do
>ratio:=Append(ratio,exact[n]/asym[n]);
>end for;
>ratio;
\end{verbatim}
}
{\color{red} 
\begin{verbatim}
[3.50000000, 2.62500000, 1.96875000, 1.43554688, 1.31250000, 1.13305664,
1.09887695, 1.04173279, 1.03121185, 1.01315403, 1.00984955, 1.00414994,
1.00310811, 1.00130949, 1.00098079]
\end{verbatim} }
\end{small}

The ratio converges to $1$ as we expect. We can paste the list of ratios into Mathematica to create the graph in Figure \ref{fig:sl2,7 4d ratio}. We briefly explain how this is done below. 

\subsection{Graphing using Mathematica}
The following Mathematica code is used to create Figure \ref{fig:sl2,7 4d ratio} based on the Magma output obtained above: we plot the ratio $b(n)/a(n)$ as well as the horizontal line valued at 1, to illustrate the geometric convergence $b(n)\sim a(n)$.

\begin{small}
{\color{blue}
\begin{verbatim}
Ratio={3.50000000, 2.62500000, 1.96875000, 1.43554688, 1.31250000, 1.13305664,
1.09887695, 1.04173279, 1.03121185, 1.01315403, 1.00984955, 1.00414994,
1.00310811, 1.00130949, 1.00098079}
Line=Table[1,15]
ListLinePlot[{Ratio,Line},PlotRange->{{0,15},{0.8,3.7}},Frame->True, 
PlotLabels -> {Style[ToExpression["b(n)/a(n)",TeXForm,HoldForm]]}]
   
\end{verbatim}
}
\end{small}


\begin{thebibliography}{LnMRM12}

\bibitem[Alp93]{alperin1993local}
J.~Alperin.
\newblock {\em Local representation theory: Modular representations as an
  introduction to the local representation theory of finite groups}.
\newblock Cambridge University Press, 1993.
\newblock \href{https://doi.org/10.1017/CBO9780511623592}{\path{doi:10.1017/CBO9780511623592}}

\bibitem[AN09]{anderson20fibonacci}
S.~Anderson and D.~Novak.
\newblock Fibonacci vector sequences and regular polygons.
\newblock {\em ResearchGate.net}, 2009.
\newblock URL: \url{https://www.researchgate.net/publication/228768410}.

\bibitem[BCP97]{bosma1997magma}
W.~Bosma, J.~Cannon, and C.~Playoust.
\newblock The Magma algebra system I: The user language.
\newblock {\em J. Symbolic Comput.}, 24(3-4):235--265, 1997.
\newblock \href{https://doi.org/10.1006/jsco.1996.0125}{\path {doi:10.1006/jsco.1996.0125}}

\bibitem[BK72]{bryant1972tensor}
R.~Bryant and L.~Kov{\'a}cs.
\newblock Tensor products of representations of finite groups.
\newblock {\em Bull. London Math. Soc.}, 4(2):133--135, 1972.
\newblock \href{https://doi.org/10.1112/blms/4.2.133}{\path{doi:10.1112/blms/4.2.133}}

\bibitem[BN41]{Brauer}
R.~Brauer and C.~Nesbitt.
\newblock On the modular characters of groups.
\newblock {\em Ann. Math.}, 42(2):556-590, 1941.
\newblock
\href{https://doi.org/10.2307/1968918}{\path{doi:10.2307/1968918}}

\bibitem[Con65]{conlon1965certain}
S.~Conlon.
\newblock Certain representation algebras.
\newblock {\em J. Aust. Math. Soc.}, 5(1):83--99, 1965.
\newblock \href{https://doi.org/10.1017/S1446788700025908}{\path{doi:10.1017/S1446788700025908}}

\bibitem[CEO24]{coulembier2023asymptotic}
K.~Coulembier, P.~Etingof, and V.~Ostrik.
\newblock Asymptotic properties of tensor powers in symmetric tensor
  categories.
\newblock {\em Pure Appl. Math. Q.}, 20(3):1141--1179, 2024.
\newblock URL: \url{https://arxiv.org/abs/2301.09804}, \href
  {https://doi.org/10.4310/PAMQ.2024.v20.n3.a4}
  {\path{doi:10.4310/PAMQ.2024.v20.n3.a4}}.

\bibitem[CEOT24]{coulembier2024fractalbehaviortensorpowers}
K.~Coulembier, P.~Etingof, V.~Ostrik, and D.~Tubbenhauer.
\newblock Fractal behavior of tensor powers of the two dimensional space in
  prime characteristic, 2024.
\newblock URL: \url{https://arxiv.org/abs/2405.16786}.

\bibitem[COT24]{Coulembier_2023}
K.~Coulembier, V.~Ostrik, and D.~Tubbenhauer.
\newblock Growth rates of the number of indecomposable summands in tensor
  powers.
\newblock {\em Algebr. Represent. Theory}, 27(2):1033--1062, 2024.
\newblock URL: \url{https://arxiv.org/abs/2301.00885}, \href
  {https://doi.org/10.1007/s10468-023-10245-7}
  {\path{doi:10.1007/s10468-023-10245-7}}.

\bibitem[Cra13]{craven2013tensor}
D.~Craven.
\newblock On tensor products of simple modules for simple groups.
\newblock {\em Algebr. Represent. Theory}, 16:377--404, 2013.
\newblock URL: \url{https://arxiv.org/abs/1102.3447}
\href{http://doi.org/10.1007/s10468-011-9311-5}{\path{doi:10.1007/s10468-011-9311-5}}.

\bibitem[Dor71]{dornhoff1971group}
L.~Dornhoff.
\newblock {\em Group representation theory: Ordinary representation theory},
  volume~7 of {\em Pure and Applied Mathematics}.
\newblock M. Dekker, 1971.

\bibitem[Gan09]{MR2500560}
T.~Gannon.
\newblock The {G}alois action on character tables.
\newblock In {\em Groups and symmetries}, volume~47 of {\em CRM Proc. Lecture
  Notes}, pages 165--172. Amer. Math. Soc., Providence, RI, 2009.
\newblock URL: \url{https://arxiv.org/abs/0710.1328}.

\bibitem[Gas54]{gaschutz1954endliche}
W.~Gasch{\"u}tz.
\newblock Endliche Gruppen mit treuen absolut-irreduziblen Darstellungen.
\newblock {\em Math. Nachr.}, 12(3-4):253--255, 1954.

\bibitem[H{\"o}l36]{Hoelder1936}
O.~H{\"o}lder.
\newblock Zur Theorie der Kreisteilungsgleichung $k_m(x)=0$.
\newblock {\em Prace Mat.-Fiz.}, 43(1):13--23, 1936.

\bibitem[Lac+24]{lachowska2024tensorpowersvectorrepresentation} A.~Lachowska, O.~Postnova, N.~Reshetikhin, and D.~Solovyev, 2024. 
\newblock Tensor powers of vector representation of $U_q(\mathfrak{sl}_2)$ at even roots of unity.
\newblock URL: \url{https://arxiv.org/abs/2404.03933}.

\bibitem[LTV23]{lacabanne2023asymptotics}
A.~Lacabanne, D.~Tubbenhauer, and P.~Vaz.
\newblock Asymptotics in finite monoidal categories.
\newblock {\em Proc. Amer. Math. Soc. Ser. B}, 10(34):398--412, 2023.
\newblock URL: \url{https://arxiv.org/abs/2307.03044}, \href
  {https://doi.org/10.1090/bproc/198} {\path{doi:10.1090/bproc/198}}.

\bibitem[LTV24]{lacabanne2024asymptotics}
A.~Lacabanne, D.~Tubbenhauer, and P.~Vaz.
\newblock Asymptotics in infinite monoidal categories, 2024.
\newblock URL: \url{https://arxiv.org/abs/2404.09513}.

\bibitem[Lar24]{larsen2024boundsmathrmsl2indecomposablestensorpowers}
M.~J.~Larsen.
\newblock Bounds for $\mathrm{SL}_2$-indecomposables in tensor powers of the
  natural representation in characteristic $2$, 2024.
\newblock URL: \url{https://arxiv.org/abs/2405.16015}.

\bibitem[LnMRM12]{roots}
J.~Lea\~{n}os, R.~Moreno, and L.~Rivera-Mart\'{\i}nez.
\newblock On the number of {$m$}th roots of permutations.
\newblock {\em Australas. J. Combin.}, 52:41--54, 2012.
\newblock URL: \url{https://arxiv.org/abs/1005.1531}.

\bibitem[Ren79]{renaud1979decomposition}
J.-C.~Renaud.
\newblock The decomposition of products in the modular representation ring of a
  cyclic group of prime power order.
\newblock {\em J. Algebra}, 58(1):1--11, 1979.
\newblock \href{https://doi.org/10.1016/0021-8693(79)90184-4}{\path{doi: 10.1016/0021-8693(79)90184-4}}

\bibitem[Ser77]{serre}
J.-P.~Serre.
\newblock {\em Linear representations of finite groups}, volume~42 of {\em
  Graduate Texts in Mathematics}.
\newblock Springer-Verlag, New York-Heidelberg, french edition, 1977.
\newblock \href{https://doi.org/10.1007/978-1-4684-9458-7}{\path{doi: 10.1007/978-1-4684-9458-7}}

\bibitem[Sri64]{Sri}B.~Srinivasan.
\newblock{On the modular characters of the special linear group $\textup{SL}(2, p^n)$}.
\newblock {\em Proc. Lond. Math. Soc.,} s3-14(1): 101-114, 1964.
\href{https://doi.org/10.1112/plms/s3-14.1.101}{\path{doi: 10.1112/plms/s3-14.1.101}}

\bibitem[Web16]{webb2016course}
P.~Webb.
\newblock {\em A course in finite group representation theory}, volume 161.
\newblock Cambridge University Press, 2016.
\newblock \href{https://doi.org/10.1017/CBO9781316677216}{\path{doi:10.1017/CBO9781316677216}}

\bibitem[Wig41]{Wigner}
E.~Wigner.
\newblock On representations of certain finite groups.
\newblock {\em Amer. J. Math.}, 63(1):57--63, 1941.
\newblock \href{https://doi.org/10.1007/978-3-662-02781-3_24}{\path{doi: 10.1007/978-3-662-02781-3_24}}

\end{thebibliography}
\end{document}